\documentclass[12pt]{article}
\usepackage{latexsym,amsmath,amssymb}

\newif\ifdviwin

\usepackage[latin1]{inputenc}
\usepackage[english]{babel}
\usepackage{indentfirst}
\usepackage[mathscr]{eucal}
\usepackage{amssymb,amsmath,amsfonts}
\usepackage{fancybox,fancyhdr}
\usepackage{graphicx}
\usepackage{color}
\usepackage[export]{adjustbox}

\newif\ifdviwin

\dviwintrue

\def\cA{\mathcal{A}}

\def\cL{\mathcal{L}}
\def\ch{\mathfrak{h}}
\def\cU{\mathcal{U}}

\let\8=\infty \let\0=\emptyset  
\let\hat=\widehat

\let\landa=\lambda
\let\alfa=\alpha

\let\parc=\partial
\let\ep=\varepsilon

\def\landa{\lambda}

\def\flecha{\rightarrow}
\def\esiz{\langle}
\def\esde{\rangle}

\def\cte.{\mathop{\rm cte.}\nolimits}

\def\div{\mathop{\rm div }\nolimits}

\def\R{\mathbb{R}}

\def\cH{\mathcal{H}}

\def\D{\mathbb{D}}

\def\S{\mathbb{S}}
\def\X{\mathfrak{X}}

 \newtheorem{defi}{Definition}
 \newtheorem{teo}[defi]{Theorem}
 \newtheorem{pro}[defi]{Proposition}
 \newtheorem{cor}[defi]{Corollary}
 \newtheorem{lem}[defi]{Lemma}
 
 \newtheorem{remark}[defi]{Remark}

 \newenvironment{proof}{\rm \trivlist \item[\hskip \labelsep{\it
      Proof}:]}{\par\nopagebreak \hfill $\Box$ \endtrivlist}

\numberwithin{equation}{section}
\numberwithin{figure}{section}
\numberwithin{defi}{section}

\begin{document}

\mbox{}\vspace{0.4cm}
\begin{center}
\rule{14cm}{1.5pt}\vspace{0.5cm}

\renewcommand{\thefootnote}{\,}
{\Large \bf The global geometry of surfaces with \\[0.2cm] prescribed mean curvature in $\R^3$} % \footnote{\hspace{-.75cm} Mathematics Subject
%Classification: 53A10, 53C42}}
\\ \vspace{0.5cm} {\large Antonio Bueno$^a$, José A.
Gálvez$^b$ and Pablo Mira$^c$}\\ \vspace{0.3cm} \rule{14cm}{1.5pt}
\end{center}
  \vspace{1cm}
$\mbox{}^a,^b$ Departamento de Geometría y Topología, Universidad de Granada,
E-18071 Granada, Spain. \\ e-mail: jabueno@ugr.es, jagalvez@ugr.es \vspace{0.2cm}

\noindent $\mbox{}^c$ Departamento de Matemática Aplicada y Estadística,
Universidad Politécnica de Cartagena, E-30203 Cartagena, Murcia, Spain. \\
e-mail: pablo.mira@upct.es \vspace{0.3cm}

 \begin{abstract}
We develop a global theory for complete hypersurfaces in $\R^{n+1}$ whose mean curvature is given as a prescribed function of its Gauss map.  This theory extends the usual one of constant mean curvature hypersurfaces in $\R^{n+1}$, and also that of self-translating solitons of the mean curvature flow. For the particular case $n=2$, we will obtain results regarding a priori height and curvature estimates, non-existence of complete stable surfaces, and classification of properly embedded surfaces with at most one end. \let\thefootnote\relax\footnote{\hspace{-.75cm} Mathematics Subject
Classification: 53A10, 53C42}
 \end{abstract}

\section{Introduction}

Let $\cH$ be a $C^1$ function on the sphere $\S^n$. We will say that an immersed oriented hypersurface $\Sigma$ in $\R^{n+1}$ has \emph{prescribed mean curvature $\cH$} if its mean curvature function $H_{\Sigma}$ is given by 
 \begin{equation}\label{presH}
H_{\Sigma}=\cH\circ \eta,
\end{equation}
where $\eta:\Sigma\flecha \S^n$ is the Gauss map of $\Sigma$. For short, we will simply say that $\Sigma$ is an $\cH$-hypersurface. Obviously, when $\cH$ is constant, $\Sigma$ is a hypersurface of constant mean curvature (CMC).

The study of hypersurfaces in $\R^{n+1}$ defined by a prescribed curvature function in terms of the Gauss map goes back, at least, to the famous Christoffel and Minkowski problems for ovaloids (see e.g. \cite{C}). The existence and uniqueness of ovaloids of prescribed mean curvature in $\R^{n+1}$ was studied among others by Alexandrov and Pogorelov in the 1950s (see \cite{Al, Po}), and more recently in \cite{GG}. However, the geometry of complete hypersurfaces of prescribed mean curvature $\cH$ in $\R^{n+1}$ remains largely unexplored. The most studied situation is, obviously, the case of complete CMC hypersurfaces in $\R^{n+1}$. Recently, the global geometry of self-translating solitons of the mean curvature flow (which correspond to $\cH$-hypersurfaces for the particular choice $\cH(x)=\esiz x,e_{n+1}\esde$) is being studied in detail, see e.g. \cite{DDN, Il, IR,MPGSHS, MSHS,SX}.

Our objective in this paper is to develop a global theory of complete $\cH$-hypersurfaces in $\R^{n+1}$, taking as a starting point the well-studied global theory of CMC hypersurfaces in $\R^{n+1}$. We will be specially interested in the case of complete $\cH$-surfaces in $\R^3$.%, which is the case for which most of our main results will be obtained; for example, the ones regarding the description of properly embedded $\cH$-surfaces of finite topology, and the geometry of stable $\cH$-surfaces in $\R^3$. %In the rest of this introduction we will explain the organization of the paper, and the

The following are three trivial but fundamental properties of $\cH$-hypersurfaces in $\R^{n+1}$, for a given $\cH\in C^1(\S^n)$: (1) any translation of an $\cH$-hypersurface is also an $\cH$-hypersurface; (2) $\cH$-hypersurfaces are locally modeled by a quasilinear elliptic PDE when viewed as local graphs in $\R^{n+1}$ over each tangent hyperplane, and consequently they obey the maximum principle; and (3) any symmetry of $\cH$ in $\S^n$ induces a linear isometry of $\R^{n+1}$ that preserves $\cH$-hypersurfaces in $\R^{n+1}$. 

These three properties make many ideas of classical CMC hypersurface theory in $\R^{n+1}$ readily available for the case of $\cH$-hypersurfaces, for a non-constant function $\cH\in C^1(\S^n)$. Nonetheless, the class of $\cH$-hypersurfaces is indeed much wider and richer, and new ideas are needed for its study. For example, $\cH$-hypersurfaces in $\R^{n+1}$ do not come in general associated to a variational problem, and this makes many of the most useful techniques of CMC theory unavailable. Also, in contrast with CMC surfaces in $\R^3$, there is no holomorphic object associated to $\cH$-surfaces in $\R^3$ for non-constant $\cH\in C^1(\S^2)$, and so, no clear integrable systems approach to their study. Moreover, in jumping from the CMC condition to the equation of prescribed non-constant mean curvature, one also needs to account for the loss of symmetries and isotropy of the resulting equation. For instance, one can only apply the Alexandrov reflection principle for directions with respect to which $\cH$ is symmetric.

Some of these difficulties can already be seen in the case of rotational $\cH$-surfaces in $\R^3$. In the paper \cite{BGM} we present many examples of rotational $\cH$-surfaces in $\R^3$ with no CMC counterpart: for instance, complete, convex $\cH$-graphs converging to a cylinder, or properly embedded disks asymptotically wiggling around a cylinder. For the case of rotational $\cH$-surfaces of non-trivial topology, we show in \cite{BGM} examples with a \emph{wing-like} shape, or with two strictly convex ends pointing in opposite directions. Many of these examples can also be constructed so that they self-intersect. All this variety, just for the very particular class of rotational $\cH$-surfaces in $\R^3$, shows that the class of $\cH$-hypersurfaces in $\R^{n+1}$ is indeed very large, and rich in what refers to possible examples and geometric behaviors. It also accounts for the difficulty of obtaining classification results for the case where $\cH$ is not constant, and justifies the need of imposing some additional regularity, positivity or symmetry properties to $\cH$ in order to derive such classification theorems.

Despite these difficulties, the results in the present paper show that, under mild assumptions, the class of $\cH$-surfaces in $\R^3$ can sometimes be given a homogeneous treatment for general classes of prescribed mean curvature functions $\cH\in C^1(\S^2)$.

In the rest of this introduction we will explain the organization of the paper, and state some of our main results.

In {\bf Section \ref{sec:2}} we will study some basic properties of $\cH$-hypersurfaces in $\R^{n+1}$. 
In Proposition \ref{hpos} we prove that if $\cH(x_0)=0$ for some $x_0\in \S^n$, then there are no compact $\cH$-hypersurfaces in $\R^{n+1}$. In Proposition \ref{ale} and Corollary \ref{3planes} we give some Alexandrov-type results using moving planes in the case that the function $\cH$ has enough symmetries. %Also, using the Alexandrov reflection principle, we prove in Proposition \ref{ale} that if $\cH\in C^1(\S^n)$ is invariant under $n$ independent geodesic reflections on $\S^n$, then any compact embedded $\cH$-hypersurface in $\R^{n+1}$ is diffeomorphic to $\S^n$. In Corollary \ref{3planes} we will obtain from this result and previous theorems in \cite{GG,GM1} that if $\cH\in C^2(\S^2)$ is positive and invariant under a group of isometries of $\S^2$ without fixed points that contains two independent reflection symmetries, then any compact embedded $\cH$-surface in $\R^3$ is a strictly convex sphere.
%A fundamental theorem by B. Guan and P. Guan (see \cite{GG}) proves that if $\cH\in C^2(\S^n)$, $\cH>0$, is invariant under a group of isometries of $\S^n$ without fixed points, then there exists a compact strictly convex $\cH$-sphere in $\R^{n+1}$, which will be called from now on the \emph{Guan-Guan $\cH$-sphere} and denoted by $S_{\cH}$. For such a choice of $\cH$, any other strictly convex $\cH$-sphere in $\R^{n+1}$ is a translation of $S_{\cH}$ (for $n=2$, this is a classical theorem by Alexandrov \cite{Al}). %, see also \cite{HW, Po}). 
%In \cite{GM1} Gálvez and Mira proved more generally that, for $n=2$, any compact $\cH$-surface of genus zero immersed in $\R^3$ is a translation of $S_{\cH}$.
In Section \ref{sec:constru} we will classify flat $\cH$-hypersurfaces. In Section \ref{sec:compac} we will prove a compactness theorem for the space of $\cH$-surfaces in $\R^3$ with bounded second fundamental form, with $\cH$ not necessarily fixed. %This result (Theorem \ref{compa}) will be key for our purposes in Sections \ref{sec:structure} and \ref{sec:stable}.

%In {\bf Section \ref{rotiper}} we consider rotational $\cH$-hypersurfaces in $\R^{n+1}$, in the case that $\cH\in C^1(\S^n)$ is \emph{rotationally symmetric}, i.e. $\cH(x)=\ch(\esiz x, e_{n+1}\esde)$ for some $\ch\in C^1([-1,1])$. In such generality for $\cH$, it seems hopeless to find an explicit description of such rotational $\cH$-hypersurfaces, similar to the cases of CMC hypersurfaces, or of self-translating solitons of the mean curvature flow. Thus, we will follow a different approach. We will treat the resulting ODE as a nonlinear autonomous system and we will carry out a qualitative study of its solutions through a phase space analysis.
%
%By using this method, we will show that any rotational $\cH$-sphere in $\R^{n+1}$ is strictly convex (Theorem \ref{teoes}). When $\cH$ vanishes at some point, we will describe in Section \ref{sec:bowls}, for very general choices of $\cH$, a family of rotational $\ch$-bowls (which are entire convex graphs) and of $\ch$-catenoids (which resemble the usual minimal catenoids in $\R^{n+1}$).
%
In {\bf Section \ref{sec:structure}} we study the geometry of properly embedded $\cH$-surfaces of finite topology in $\R^3$, motivated by some well-known results by Meeks \cite{Me} and Korevaar-Kusner-Solomon \cite{KKS} for the case where $\cH$ is a positive constant. In \cite{Me}, Meeks proved that there are no properly embedded surfaces of positive CMC in $\R^3$ with finite topology and exactly one end. For that, he first obtained universal height estimates for (not necessarily compact) CMC graphs in $\R^3$ with planar boundary. In \cite{KKS} it was proved that any CMC surface in $\R^3$ in the conditions above, but this time with \emph{two} ends, is a rotational surface.

For the general case of $\cH$-surfaces in $\R^3$ with $\cH\in C^1(\S^2)$, $\cH>0$, these statements are not true in general. However, our main results in Section \ref{sec:structure} give natural additional conditions on $\cH$ under which Meeks' results hold for $\cH$-surfaces.

Specifically, in Theorem \ref{th:hees} we will give some very general sufficient conditions on $\cH$ for the existence of a universal height estimate for $\cH$-graphs with respect to a given direction $v\in \S^2$ in $\R^3$. %, the most useful {\color{blue}of these} for our purposes being the existence of an \emph{$\cH$-hemisphere} {\color{blue}with respect to that direction $v$}. 
The way of proving these height estimates is completely different from the one used by Meeks, and relies on a previous curvature estimate for surfaces in $\R^3$ in the spirit of a result on CMC surfaces in Riemannian three-manifolds by Rosenberg, Sa Earp and Toubiana \cite{RST}; see Theorem \ref{th:curv}.

As corollaries to Theorem \ref{th:hees}, in Section \ref{estru} we will provide several results about the geometry of properly embedded $\cH$-surfaces in $\R^3$ of finite topology and one end, in the case that $\cH>0$ is invariant under one, two or three reflections in $\S^2$. For instance, in the case that $\cH$ is invariant with respect to three linearly independent reflections in $\S^2$, we will prove (Theorem \ref{3sim}) that any properly embedded $\cH$-surface in $\R^3$ of finite topology and at most one end is the Guan-Guan sphere $S_{\cH}$ associated to $\cH$, see \cite{GG}.

In {\bf Section \ref{sec:stable}} we study \emph{stability} of $\cH$-hypersurfaces. As we already mentioned, except for some very particular cases, the equation describing $\cH$-hypersurfaces in $\R^{n+1}$ is \emph{not} the Euler-Lagrange equation of some variational problem. Thus, the way in which we introduce the concept of stability for $\cH$-hypersurfaces (Definition \ref{def:stabH}) is not in a variational way, but in connection with the \emph{linearized equation} of the $\cH$-graph equation in $\R^{n+1}$; see Proposition \ref{varcar}. In this way we obtain, for any $\cH$-hypersurface $\Sigma$ in $\R^{n+1}$ a linear \emph{stability operator} of the form 

 \begin{equation}\label{opL}
\cL=\Delta + \esiz \nabla \cdot, X\esde + q
 \end{equation} for some $X\in \X(\Sigma)$ and $q\in C^{2}(\Sigma)$. 
 
 %After defining stability associated to $\cL$ in the natural way, we will  ob
 
 %With this notion at hand, we will show that it is natural to define, in analogy with the CMC case, a \emph{stable $\cH$-hypersurface} $\Sigma$ in $\R^{n+1}$ as one for which there exists $u\in C^{2}(\Sigma)$, $u>0$, with $\cL(u)\leq 0$. This definition is consistent not only with the stability notion in CMC hypersurface theory, but also with the ones in the theories of self-translating solitons of the mean curvature flow, and of marginally outer-trapped surfaces (MOTS). This stability notion also implies the non-negativity of the principal eigenvalue of the (non self-adjoint) operator $-\cL$. One easy consequence of this notion will be that \emph{there are no compact (without boundary) stable $\cH$-hypersurfaces in $\R^{n+1}$}; see Corollary \ref{conoco}.

After defining stability associated to $\cL$ in the natural way, we will focus on stable $\cH$-surfaces $\Sigma$ in $\R^3$, and will seek radius and curvature estimates for $\Sigma$. In the CMC case, Ros and Rosenberg proved in \cite{RoRos} that the intrinsic distance from a point $p$ in a stable CMC surface $\Sigma$ to $\partial \Sigma$ is less than $\pi/H$, where $H>0$ is the mean curvature of $\Sigma$. In \cite{Maz}, Mazet improved the previous result obtaining the optimal estimate $\pi /(2H)$.

Our approach to proving this type of radius estimates for stable $\cH$-surfaces will be different, and based on arguments introduced by Fischer-Colbrie \cite{FC} and López-Ros \cite{LR} for non-negative \emph{Schrodinger operators}, i.e. operators of the form \eqref{opL} for $X=0$. We will also make use of an argument by Galloway and Schoen \cite{GaSc}, which will let us bound the (non-Schrodinger) stability operator $\cL$ in \eqref{opL} by a non-negative Schrodinger-type one; see Lemma \ref{lemgs}.

By using these ideas, we will prove (Theorem \ref{deste}) that there exists a uniform radius estimate for stable $\cH$-surfaces in $\R^3$, in the case that $\cH\in C^2(\S^2)$ is positive and satisfies a certain additional condition, see inequality \eqref{estrella}. As a trivial consequence we obtain (Corollary \ref{cor:estrella}): \emph{if $\cH\in C^2(\S^2)$ satisfies condition \eqref{estrella}, there are no complete stable $\cH$-surfaces in $\R^3$.}

It is important to observe here that some condition on $\cH>0$ is needed in order to prove a radius estimate for stable $\cH$-surfaces. For instance, for some radially symmetric choices of $\cH>0$, there exist complete, non-entire, strictly convex $\cH$-graphs in $\R^3$.%; see, e.g., the example in Figure \ref{fig:grim}.

Also, in Theorem \ref{testicu} we will prove a curvature estimate, of the form $$
|\sigma(p)|d_{\Sigma}(p,\partial \Sigma)\leq C,
$$ for stable $\cH$-surfaces $\Sigma$ in $\R^3$, where $\cH\in C^2(\S^2)$ satisfies \eqref{estrella}; here $\sigma$ is the second fundamental form of $\Sigma$. For the case of minimal surfaces and CMC surfaces in $\R^3$, such estimates were obtained by Schoen \cite{Sc} and Berard-Hauswrith \cite{BH}, respectively.

\vspace{0.3cm}

{\bf Acknowledgements:} This work is part of the PhD thesis of the first author. The authors are grateful to José M. Espinar, Francisco Martín, Joaquín Pérez and Francisco Torralbo for helpful comments during the preparation of this manuscript. 
\section{Basic properties of $\cH$-hypersurfaces}\label{sec:2}

\subsection{Maximum and tangency principles}

Locally, $\cH$-hypersurfaces in $\R^{n+1}$ for some given $\cH\in C^1(\S^n)$ are governed by an elliptic, second order quasilinear PDE. Specifically, let $\cH\in C^1(\S^n)$, let $\Sigma$ denote an $\cH$-hypersurface, and take any $q\in \Sigma$. Then, if we view $\Sigma$ around $q$ as an upwards-oriented graph $x_{n+1}=u(x_1,\dots, x_n)$ with respect to coordinates $x_i$ in $\R^{n+1}$ such that $\{\parc_1,\dots,\parc_{n+1}\}$ is a positively oriented orthonormal frame of $\R^{n+1}$, equation \eqref{presH} shows that the function $u$ is a solution to 

\begin{equation}\label{eqH}
{\rm div}\left(\frac{Du}{\sqrt{1+|Du|^2}}\right) = n \cH (Z_u), \hspace{1cm} Z_u:= \frac{(-Du,1)}{\sqrt{1+|Du|^2}},
\end{equation}
where $\div, D$ denote respectively the divergence and gradient operators on $\R^n$. In particular, $\cH$-hypersurfaces satisfy the Hopf maximum principle (both in its interior and boundary versions), a property that we will use geometrically in the following way:%, {\color{blue} and call the \emph{tangency principle}:}

\begin{lem}\label{tan}
%Let $u_1,u_2$ denote two solutions to \eqref{eqH} on a compact domain $\Omega\subset \R^2$, with $u_1=u_2$ and $Du_1=Du_2$ at some $p\in \Omega$. Assume that $u_1\leq u_2$ on a neighborhood of $p$ in $\Omega$. Then $u_1=u_2$ in $\Omega$.
%
Given $\cH\in C^1(\S^n)$, let $\Sigma_1,\Sigma_2$ be two embedded $\cH$-hypersurfaces, possibly with smooth boundary. Assume that one of the following two conditions holds:
\begin{enumerate}
\item
There exists $p\in {\rm int}(\Sigma_1)\cap {\rm int}(\Sigma_2)$ such that $\eta_1(p)=\eta_2(p)$, where $\eta_i:\Sigma_i\flecha \S^n$ is the unit normal of $\Sigma_i$, $i=1,2$.
\item
There exists $p\in \parc \Sigma_1 \cap \parc \Sigma_2$ such that $\eta_1(p)=\eta_2(p)$ and $\nu_1(p)=\nu_2(p)$, where $\nu_i$ denotes the interior unit conormal of $\parc \Sigma_i$.
\end{enumerate}
Assume moreover that $\Sigma_1$ lies around $p$ at one side of $\Sigma_2$. Then $\Sigma_1=\Sigma_2$.
\end{lem}

In the case that $\cH$ is an odd function, i.e. $\cH(-x)=-\cH(x)$ for every $x\in \S^n$, any $\cH$-hypersurface is also an $\cH$-hypersurface with the opposite orientation. Hence, the tangency principle stated in Lemma \ref{tan} can be formulated in a stronger way, similar to the usual tangency principle for minimal hypersurfaces. In particular, we have:

\begin{cor}\label{tangency}
Let $\cH\in C^1(\S^n)$ satisfy $\cH(-x)=-\cH(x)$ for every $x\in \S^n$, and let $\Sigma_1,\Sigma_2$ denote two immersed $\cH$-hypersurfaces in $\R^{n+1}$. Assume that there exists $p\in \Sigma_1\cap \Sigma_2$ with $T_p\Sigma_1=T_p\Sigma_2$ and such that $\Sigma_1$ lies around $p$ at one side of $\Sigma_2$. Then $\Sigma_1=\Sigma_2$.
\end{cor}

%In local arbitrary coordinates $(x,y,z)$, the equation $H_{\Sigma}=\cH\circ \eta$ can be written as an elliptic PDE when viewed as an equation for downwards-oriented graphs. For instance, when $(x,y,z)$ are the canonical Euclidean coordinates and $\Sigma$ is a graph $z=u(x,y)$, this equation is written as 
%$$
%\frac{\displaystyle{(1+u_x^2)u_{yy}-2u_xu_yu_{xy}+(1+u_y^2)u_{xx}}}{(1+|\nabla u|^2)^\frac{3}{2}}=2\cH\Bigg(\frac{1}{\sqrt{1+|\nabla u|^2}}\big(-\nabla u,1\big)\Bigg),
%$$ 

%In particular, the maximum principle for PDE's can be applied for this family of surfaces.  This tool is very useful when applied using isometries of the ambient space that preserve the problem. A crucial step is to move the surface by these isometries while preserving its properties. In this sense, the $\cH$-surfaces do not behave well with respect to euclidean isometries, as they are only invariant with respect to translations; any other isometry changes the Gauss map and hence the value of the mean curvature.

It is immediate that any translation of an $\cH$-hypersurface is again an $\cH$-hypersurface. Moreover, the possible symmetries of the function $\cH$ induce further invariance properties of the class of $\cH$-hypersurfaces, as detailed in the next lemma:

\begin{lem}\label{sime}
Given $\cH\in C^1(\S^n)$, let $\Sigma$ be an $\cH$-hypersurface, and let $\Phi$ be a linear isometry of $\R^{n+1}$ such that $\cH \circ \Phi =\cH$ on $\S^n\subset \R^{n+1}$.
Then $\Sigma'=\Phi \circ \Sigma$ is an $\cH$-hypersurface in $\R^{n+1}$ with respect to the orientation given by $\eta':=d\Phi(\eta)$.
\end{lem}

\begin{remark}\label{lefti}
The notion of a surface of prescribed mean curvature $\cH$ can be generalized to the context of oriented hypersurfaces in $n$-dimensional Lie groups with a left invariant metric. Specifically, if $\psi:M^n\flecha G^{n+1}$ is a hypersurface in such a Lie group $G$, we can define its \emph{left invariant Gauss map} $\eta: M\flecha \S^n = \{v\in \mathfrak{g} : |v|=1\}$ by left translating to the identity element $e$ of $G$ the unit normal vector of $\psi$ at every point. Thus, given $\cH\in C^1(\S^n)$, we may define a \emph{hypersurface of prescribed mean curvature $\cH$} by equation \eqref{presH}. See e.g. \cite{GM1,GM3}.

%In this work we will only focus on the case where the ambient space is $\R^{n+1}$; some properties of surfaces of prescribed mean curvature in three-dimensional metric Lie groups can be found in \cite{GM1,GM2}.
\end{remark}

%If we suppose that $\cH$ is rotationally invariant in $\S^2$, i.e., we consider a $\ch$-surface, then the family of isometries that preserve the angle function grows, as we can see in the following proposition: 

\subsection{A particular case: self-translating solitons of mean curvature flows}\label{solit}

We consider next the particular case in which $\cH\in C^1(\S^n)$ is a linear function, i.e. $\cH(x)= a \esiz x,v\esde +b$ for $a,b\in \R$, $a\neq 0$, and $v\in \S^n$. Up to a homothety and a change of Euclidean coordinates in $\R^{n+1}$, we can assume that $a=1$ and that $v=e_{n+1}$, and so 
 \begin{equation}\label{hlin}
 \cH(x)=\esiz x,e_{n+1}\esde +b.
 \end{equation}
Let $f:\Sigma\flecha \R^{n+1}$ be an immersed oriented $\cH$-hypersurface for $\cH$ as in \eqref{hlin}, let $\eta:\Sigma\flecha \S^n$ be its unit normal, and consider the family of translations of $f$ in the $e_{n+1}$ direction, given by $F(p,t)= f(p)+t e_{n+1}$. Then, $F(p,t)$ is a solution to the geometric flow below, which corresponds to the mean curvature flow with a constant forcing term:

\begin{equation}\label{geoflo}
\left(\frac{\parc F}{\parc t}\right)^{\perp}= ( H - b)\eta,
\end{equation}
where $H=H(\cdot,t)$, $\eta=\eta(\cdot,t)$ denote the mean curvature and unit normal of the hypersurface $F(\cdot, t):\Sigma\flecha \R^{n+1}$. In other words, $f:\Sigma\flecha \R^{n+1}$ is a \emph{self-translating soliton} of the geometric flow \eqref{geoflo}. Note that when $b=0$, \eqref{geoflo} is the usual mean curvature flow (in this paper we use the convention for $H$ that the mean curvature of the unit sphere in $\R^{n+1}$ with its inner orientation is equal to $1$). The converse also holds, i.e. any self-translating soliton to \eqref{geoflo} is an $\cH$-hypersurface in $\R^{n+1}$ for $\cH$ as in \eqref{hlin}.

There is a second characterization for $\cH$-hypersurfaces with $\cH$ given by \eqref{hlin}. For that, recall following Gromov \cite{Gr} (see also \cite{BCMR}), that the \emph{weighted mean curvature} $H_{\phi}$ of an oriented hypersurface $\Sigma$ in $\R^{n+1}$ with respect to the density $e^{\phi}\in C^1 (\R^{n+1})$ is given by 
 \begin{equation}\label{wmc}
 H_{\phi}= n H_{\Sigma} - \esiz \eta,D \phi\esde,
 \end{equation}
where $H_{\Sigma}$ is the mean curvature of $\Sigma$ in $\R^{n+1}$, $\eta$ is its unit normal, and $D$ denotes the gradient in $\R^{n+1}$. So, by \eqref{wmc} and the previous discussion we arrive at:

\begin{pro}\label{carwe}
Let $\Sigma$ be an immersed oriented hypersurface in $\R^{n+1}$. The following three conditions are equivalent:
\begin{enumerate}
\item
$\Sigma$ is an $\cH$-hypersurface, for $\cH$ given by \eqref{hlin}.
 \item
$\Sigma$ is a self-translating soliton of the mean curvature flow with a constant forcing term \eqref{geoflo}.
 \item
$\Sigma$ has constant weighted mean curvature $H_{\phi}= nb \in \R$ for the density $e^{\phi}$ in $\R^{n+1}$, where $\phi(x):=n \esiz x,e_{n+1}\esde$.
\end{enumerate}
\end{pro}
%
%The class of $\cH$-hypersurfaces in $\R^{n+1}$ for $\cH$ of the form \eqref{hlin} is also related to the \emph{volume preserving mean curvature flow} introduced by Huisken \cite{Hu2}. Specifically, for $\Sigma$ compact (maybe with non-empty boundary) in $\R^{n+1}$, this flow is given, with the notation of \eqref{geoflo}, by 
% \begin{equation}\label{vpmcf}
%\left(\frac{\parc F}{\parc t}\right)^{\perp} = (H-\overline{H}) \eta,
% \end{equation} 
%where $$\overline{H}=\overline{H}(t)=\frac{1}{V_t} \int_{\Sigma_t} H(\cdot,t) dV_t,$$ where $\Sigma_t=F(\cdot, t)$ and $V_t$ is the volume of $\Sigma_t$. 
%
%Hence, in the particular case that $\Sigma_t =\Sigma + t e_{n+1}$, $\overline{H}(t)$ is a constant $b$, and any $\cH$-hypersurface $\Sigma$ with $\cH(x)$ given by \eqref{hlin} is a self-translating soliton for \eqref{vpmcf}.

As explained in the introduction, there are many relevant works on self-translating solitons of the mean curvature flow (which corresponds to the case $b=0$ in \eqref{hlin}). The situation when $b\neq 0$ in \eqref{hlin} is much less studied. Some particular references for this situation are \cite{Es,Lo}.

%See for example \cite{CSS, Hu1, Hu1, HS, Il, MPGSHS, MSHS}.

\subsection{Compact $\cH$-hypersurfaces in $\R^{n+1}$}

\begin{pro}\label{hpos}
Let $\cH\in C^1(\S^n)$, and assume that $\cH(x_0)=0$ for some $x_0\in \S^n$. Then, there are no compact $\cH$-hypersurfaces in $\R^{n+1}$.
\end{pro}
\begin{proof}
Take $\cH\in C^1(\S^n)$ with $\cH(x_0)=0$ for some $x_0\in \S^n$, and let $\Sigma$ denote a compact immersed $\cH$-hypersurface in $\R^{n+1}$. Let $\Pi$ denote an oriented hyperplane in $\R^{n+1}$ with unit normal $x_0$, and such that $\Sigma$ is contained in the half-space $\Pi^+:=\cup \{\Pi_{\landa} : \landa>0\}$, where $\Pi_{\landa}:=\Pi + \landa x_0$. Let $\landa_0$ denote the smallest $\landa>0$ for which $\Sigma\cap \Pi_{\landa} \neq \emptyset$, and let $p\in \Sigma\cap \Pi_{\landa_0}$ be any of such \emph{first contact points}. Note that $\eta(p)=\pm x_0$, where $\eta:\Sigma\flecha \S^n$ is the unit normal to $\Sigma$.

If $\eta(p)=x_0$, then both $\Sigma,\Pi_{\landa_0}$ are $\cH$-hypersurfaces and we obtain a contradiction with the maximum principle (Lemma \ref{tan}). Hence, $\eta(p)=-x_0$. Consequently, the principal curvatures $\kappa_i(p)$ of $\Sigma$ at $p$ are $\leq 0$, and hence $\cH(-x_0)\leq 0$. As a matter of fact, $\cH(-x_0)<0$, for if $\cH(-x_0)=0$, then $\Pi_{\landa_0}$ with its opposite orientation is also an $\cH$-hypersurface, and we contradict again Lemma \ref{tan}.

Let now $\landa_1$ be the largest $\landa>0$ for which $\Sigma\cap \Pi_{\landa}\neq \emptyset$, and take $q\in \Sigma\cap \Pi_{\landa_1}$. By previous arguments, $\eta(q)=-x_0$. Since $H_{\Sigma}(q)<0$ (because $\cH(-x_0)<0$), one of the principal curvatures $\kappa_i(q)$ of $\Sigma$ at $q$ is negative. This is a contradiction with the fact that $\Sigma\cap \Pi_{\landa}$ is empty for all $\landa>\landa_1$. This contradiction proves Proposition \ref{hpos}
\end{proof}

By Proposition \ref{hpos}, we see that given $\cH\in C^1(\S^n)$, a necessary condition for the existence of a compact $\cH$-hypersurface $\Sigma$ is that $\cH>0$ or $\cH<0$ on $\S^n$ (up to a change of orientation, we can assume $\cH>0$ in these conditions). However, this is \emph{not} a sufficient condition. To see this, let $\psi:\Sigma\flecha \R^{n+1}$ denote a compact immersed hypersurface of prescribed mean curvature $\cH\in C^1(\S^n)$, take $v\in \S^n$ and denote $h=\esiz \psi,v\esde$. It is then well known that $\Delta h = n H_{\Sigma} \esiz \eta,v\esde$, where $\eta$ denotes the unit normal of $\Sigma$ and $\Delta$ is the Laplacian on $\Sigma$. By the divergence theorem, and since $H_{\Sigma}=\cH\circ \eta$, we have
 \begin{equation}\label{intcon}
\int_\Sigma \langle\eta,v\rangle \, \cH(\eta) =0.
\end{equation}
From \eqref{intcon}, and noting that $\int_{\Sigma} \esiz \eta,v\esde =0$ for every $v\in \S^n$ if $\Sigma$ is compact, we have:

\begin{cor}\label{cornex}
Let $\cH\in C^1(\S^n)$ be of the form $\cH(x)=h_0(x)+b$, where $b\in \R$ and $h_0\in C^1(\S^n)$ is not identically zero and satisfies $h_0(x)\esiz x,v\esde \geq 0$ for every $x\in \S^n$ and for some $v\in \S^n$.

Then, there are no compact $\cH$-hypersurfaces in $\R^{n+1}$.
\end{cor}
For the choice $\cH(x)=\esiz x,v\esde +b$, Corollary \ref{cornex}  %, i.e. for hypersurfaces of constant weighted mean curvature with respect to a density $e^{\phi}$ with $\phi$ linear (see Proposition \ref{carwe}), 
is due to López, see \cite{Lo}.

Lemmas \ref{tan} and \ref{sime} allow to use the Alexandrov reflection principle in any direction in $\R^{n+1}$ with respect to which the function $\cH$ is symmetric. We singularize the following consequence.

\begin{pro}\label{ale}
Let $\cH\in C^1(\S^n)$ be invariant with respect to $n$ linearly independent geodesic reflections $T_1,\dots, T_n$ of $\S^n$. Then, any compact embedded $\cH$-hypersurface in $\R^{n+1}$ is diffeomorphic to $\S^n$. Moreover, $\Sigma$ is a symmetric bi-graph with respect to hyperplanes $\hat{\Pi}_1,\dots, \hat{\Pi}_n$ parallel to the hyperplanes $\Pi_i$ of $\R^{n+1}$ fixed by $T_i$.
\end{pro}
\begin{proof}
% {\bf \color{blue} (Escribirlo en dimensión arbitraria)}
%The argument is an adaptation of an idea of Rosenberg for constant mean curvature surface in the homogeneous manifold ${\rm Sol}_3$; see [...].
Let $v_1,\dots, v_n$ be unit normal vectors to $\Pi_1,\dots, \Pi_n$. By applying the Alexandrov reflection technique with respect to these directions, we conclude using Lemma \ref{tan} and Lemma \ref{sime} that, for each $i$, $\Sigma$ is a symmetric bi-graph with respect to some hyperplane $\hat{\Pi}_i$ of $\R^{n+1}$ parallel to $\Pi_i$. Define now $\Gamma:=\cap_{i=1}^n \hat{\Pi}_i$, which is a straight line in $\R^{n+1}$. For definiteness, we will assume that $\Gamma$ is the $x_{n+1}$-axis, and so each $v_i$ is horizontal, i.e. $\esiz v_i,e_{n+1}\esde =0$.

Let $P_t =\{x_{n+1}=t\}\subset \R^{n+1}$. By compactness of $\Sigma$, there is a smallest interval $[a,b]\subset \R$ such that $P_t\cap \Sigma=\emptyset$ if $t\not\in [a,b]$. Also, as $\Sigma$ is a (connected) symmetric bi-graph with respect to the horizontal directions $v_1,\dots, v_n$, it is clear that $P_a\cap \Sigma=\{({\bf 0},a)\}$ and $P_b\cap \Sigma= \{({\bf 0},b)\}$, and that $P_t\cap \Sigma$ is non-empty and transverse for all $t\in (a,b)$. %In particular, $P_t\cap \Sigma$ is connected.

Let $\Omega\subset \R^{n+1}$ be the domain bounded by $\Sigma$, and define $\Omega_t:=\Omega\cap P_t$ for $t\in (a,b)$. Note that for $t<b$ sufficiently close to $b$, $\Omega_t$ is diffeomorphic to an $n$-dimensional ball $\mathbb{B}_n$. As the intersection $P_t\cap \Sigma$ is transverse for all $t\in (a,b)$, we deduce that $\Omega_t$ is diffeomorphic to $\mathbb{B}_n$ for all $t\in (a,b)$. This implies directly that $\Omega$ is diffeomorphic to $\mathbb{B}_{n+1}$ and so, that $\Sigma$ is diffeomorphic to $\S^n$.

\end{proof}

Proposition \ref{ale} is an extension to $\cH$-hypersurfaces of Alexandrov's theorem according to which compact embedded CMC hypersurfaces in $\R^{n+1}$ are round spheres. The other fundamental result for compact CMC surfaces in $\R^3$ is Hopf's theorem (i.e. CMC surfaces diffeomorphic to $\S^2$ immersed in $\R^3$ are round spheres). An extension of Hopf's theorem to $\cH$-surfaces in $\R^3$ follows from work by the second and third authors \cite{GM1,GM2}:

\begin{teo}[\cite{GM1}]\label{hopf}
Let $\cH\in C^1(\S^2)$, $\cH>0$, and assume that there exists a strictly convex $\cH$-sphere $S$ in $\R^3$. Then any compact immersed $\cH$-surface of genus zero is a translation of $S$.
\end{teo}

In \cite{GG} B. Guan and P. Guan proved the following result:

\begin{teo}[\cite{GG}]\label{tgg}
Let $\cH\in C^2(\S^n)$ be positive and invariant under a group of isometries of $\S^n$ without fixed points. Then there exists a closed strictly convex $\cH$-hypersurface in $\R^{n+1}$.
\end{teo}
In particular, if $\cH\in C^2(\S^n)$ satisfies $\cH(x)=\cH(-x)>0$ for all $x\in \S^n$, then there exists a closed, strictly convex $\cH$-hypersurface $S_{\cH}$ in $\R^{n+1}$. When $n=2$, the Guan-Guan sphere $S_{\cH}$ is actually the unique immersed $\cH$-sphere in $\R^3$, by Theorem \ref{hopf}. 

%A particular case of functions $\cH$ in these conditions is given in the following definition:

%This comes from the fact that the previous isometries does not change de value of the angle function, and thus the mean curvature at any point is unchanged.

%\begin{defi}
%We say that $\cH\in C^1(\S^2)$ has \emph{three reflection symmetries} if there exist three mutually orthogonal geodesics $\gamma_i$ in $\S^2$, $i=1,2,3$, such that $\cH\circ T_i =\cH$ where $T_i$ is the symmetry in $\S^2$ that fixes $\gamma_i$ pointwise.
%\end{defi}
By combining Proposition \ref{ale}, Theorem \ref{hopf} and Theorem \ref{tgg}, we then have:

\begin{cor}\label{3planes}
Let $\Sigma$ be a compact embedded $\cH$-surface in $\R^3$, where $\cH\in C^2(\S^2)$ is positive and invariant under a group of isometries of $\S^2$ without fixed points that contains two geodesic reflections of $\S^2$. Then $\Sigma$ is a translation of the Guan-Guan sphere $S_{\cH}$.
\end{cor}

Examples of groups of isometries of $\S^2$ in the conditions of Corollary \ref{3planes} are those generated by reflections with respect to three linearly independent geodesics of $\S^2$. In particular, the groups of isometries of $\S^2$ that leave invariant a Platonic solid inscribed in $\S^2$ are in the conditions of Corollary \ref{3planes}.

\subsection{Construction of flat $\cH$-hypersurfaces}\label{sec:constru}

Let $\Sigma$ be a complete flat $\cH$-hypersurface, i.e. $\Sigma$ is the form $\Sigma = \alfa\times \R^{n-1}$, where $\alfa$ is a complete regular curve in a two-dimensional plane $\Pi\equiv \R^2\subset \R^{n+1}$ (here, we denote $\R^{n-1}\equiv \Pi^{\perp}$). It follows from \eqref{presH} that $\alfa$ satisfies

\begin{equation}\label{planmin}
\kappa_{\alfa} = \frac{1}{n} \cH({\bf n}),
\end{equation}
where $\kappa_{\alfa}$, ${\bf n}$ denote, respectively, the geodesic curvature and unit normal of the planar curve $\alfa$.
This equation might be seen as an analogous of the planar Minkowski problem; note, however, that in our case $\cH$ is not assumed to be positive, and $\alfa$ is not necessarily closed.

Let us consider then $\cH\in C^1(\S^n)$, let $\{e_1,e_2\}$ be a positively oriented orthonormal basis of $\Pi$, and define a $2\pi$-periodic function $\hat{\cH} \in C^1(\R)$ by 
$$\hat{\cH}(\theta):= \frac{1}{n} \cH(-\sin \theta e_1 + \cos \theta e_2).$$
Fix $v=\cos \theta_0 e_1 + \sin \theta_0 e_2\in \Pi$, for some $\theta_0$. If $\hat{\cH}(\theta_0)=0$, the straight line generated by $v$ solves \eqref{planmin}, and the corresponding $\cH$-hypersurface $\Sigma$ is a hyperplane in $\R^{n+1}$.

Suppose now that $\hat{\cH}(\theta_0)\neq 0$, and let $I_0\subset \R$ denote the largest interval containing $\theta_0$ where $\hat{\cH}$ does not vanish. Then we may consider $F(x)\in C^2(I_0)$ to be a primitive of $1/\hat{\cH}(x)$ in $I_0$, with $F(\theta_0)=0$. By periodicity of $\hat{\cH}$, we have $F'\geq c>0$ for some $c$. If $\hat{\cH}>0$ everywhere, then $I_0=\R$ and we can define the inverse function $F^{-1}$ of $F$, which is globally defined on $\R$. If $\hat{\cH}=0$ somewhere, then $I_0$ is a bounded open interval $(a,b)$ in $\R$, and $F'(x)\to \8$ as $x\to \{a,b\}$. Thus, the same conclusion for $F^{-1}$ holds.

If we write now $\alfa'(s)=\cos \theta(s) e_1 + \sin \theta(s) e_2$ for an arclength parametrization $\alfa(s)$ of $\alfa$, equation \eqref{planmin} is rewritten as 
 \begin{equation}\label{planmin2}
 \theta'(s)= \hat{\cH} (\theta(s)),
 \end{equation}
that, with the initial condition $\theta(0)=\theta_0$, has as unique solution $\theta(s)= F^{-1}(s):\R\flecha I_0$. By construction, $\alfa(s)$ is complete.

Thus, we arrive at the following result.

\begin{pro}\label{cch}
Let $\cH\in C^1(\S^n)$. Fix $v\in \S^n$ and a two-dimensional linear subspace $\Pi\equiv \R^2$ in $\R^{n+1}$ with $v\in \Pi$. Then, there exists a unique (up to translation) complete regular curve $\alfa=\alfa_{v}$ in $\R^2$ with the following properties:
 \begin{enumerate}
 \item
$\Sigma_{v,\Pi}:=\alfa \times \R^{n-1}$ is a (complete, flat) $\cH$-hypersurface in $\R^{n+1}$.
 \item
$v$ is tangent to $\Sigma_{v,\Pi}$ at some point.
 \end{enumerate}
Conversely, any complete flat $\cH$-hypersurface in $\R^{n+1}$ is one of the examples $\Sigma_{v,\Pi}$.
\end{pro}

%One should observe that the $\cH$-hypersurface $\Sigma_{v,\Pi}$ above can be explicitly constructed from the restriction of $\cH$ to the geodesic $\S^n\cap \Pi$ of $\S^n$, following the process described just before Proposition \ref{cch}.

The hypersurface $\Sigma_{v,\Pi}=\alfa \times \R^{n-1}$ is diffeomorphic to $\S^1\times \R^{n-1}$ or to $\R^n$, depending on whether $\alfa$ is a closed curve or not. By construction, a necessary condition for $\alfa$ to be closed is that $\cH$ restricted to $\S^n\cap \Pi$ never vanishes. Thus, the next result follows directly from the classical solution to Minkowski problem for planar curves:

\begin{cor}\label{difci}
Given $\cH\in C^1(\S^n)$, let $\Sigma_{v,\Pi}=\alfa\times \R^{n-1}$ be one of the $\cH$-hypersurfaces in $\R^{n-1}$ constructed in Proposition \ref{cch}. The next two conditions are equivalent:
 \begin{enumerate}
 \item
$\Sigma_{v,\Pi}$ is diffeomorphic to $\S^1\times \R^{n-1}$ (i.e. $\alfa$ is closed).
 \item
If we denote $S^1:=  \S^n\cap \Pi$, then $\cH(\xi) \neq 0$ for every $\xi \in S^1$, and $$\int_{S^1} \frac{\xi}{\cH(\xi)} d\xi =0.$$
 \end{enumerate}
\end{cor}

\subsection{A compactness theorem}\label{sec:compac}

We introduce next a compactness theorem for the space of $\cH$-surfaces in $\R^3$ with bounded second fundamental form %(or, equivalently, with bounded curvature) 
that will be used in later sections. The argument in the next theorem is well-known for CMC surfaces, see e.g. Section 2 in \cite{RST}. We sketch it here for the case where $\cH$ is not constant.

\begin{teo}\label{compa}
Let $(\Sigma_n)_n$ be a sequence of $\cH_n$-surfaces in $\R^3$ for some sequence of functions $\cH_n\in C^k(\S^2)$, $k\geq 1$, and take $p_n\in \Sigma_n$. Assume that the following conditions hold:
 \begin{enumerate}
 \item[(i)]
There exists a sequence of positive numbers $r_n\to \8$ such that the geodesic disks $D_n= D_{\Sigma_n} (p_n, r_n)$ are contained in the interior of $\Sigma_n$, i.e. $d_{\Sigma_n} (p_n,\parc \Sigma_n) \geq r_n$.
 \item[(ii)]
$(p_n)_n \to p$ for a certain $p\in \R^3$.
 \item[(iii)]
If $|\sigma_n |$ denotes the length of the second fundamental form of $\Sigma_n$, then there exists $C>0$ such that $| \sigma_n | (x) \leq C$ for every $n$ and every $x \in \Sigma_n$.
 \item[(iv)]
$\cH_n \to \cH$ in the $C^k$ topology to some $\cH \in C^k(\S^n)$.
 \end{enumerate}
Then, there exists a subsequence of $(\Sigma_n)_n$ that converges uniformly on compact sets in the $C^{k+2}$ topology to a complete, possibly non-connected, $\cH$-surface $\Sigma$ of bounded curvature that passes through $p$.
\end{teo}
\begin{proof}
By conditions (i), (iii) and by virtue of a well-known result in surface theory (see e.g. Proposition 2.3 in \cite{RST}), there exist positive constants $\delta,M$ that only depend on $C$ (and not on $n$, $\cH_n$ or $\Sigma_n$), such that, if $n$ is large enough:
 \begin{enumerate}
 \item[a)]
An open neighbourhood of $p_n$ in $D_n\subset \Sigma_n$ is the graph of a function $u_n$ over the Euclidean disk $D_{\delta}:=D(0,\delta)$ of radius $\delta$ in $T_{p_n}\Sigma_n$.
 \item[b)]
The $C^2$ norm of $u_n$ in $D_{\delta}$ is not greater than $M$.
 \end{enumerate}
Since $\Sigma_n$ is an $\cH_n$-surface, it follows that, in adequate Euclidean coordinates $(x^n,y^n,z^n)$ with respect to which $T_{p_n}\Sigma_n= \{z^n=0\}$, each function $u_n$ is a solution in $D_{\delta}$ to the quasilinear elliptic equation
 \begin{equation}\label{eqhn}
 {\rm div}\left(\frac{Du}{\sqrt{1+|Du|^2}}\right) = 2 \cH_n (Z_u), \hspace{1cm} Z_u:= \frac{(-Du,1)}{\sqrt{1+|Du|^2}}.
 \end{equation}
Note that \eqref{eqhn} for $u_n$ can be rewritten as a linear elliptic PDE $L[u_n]=f_n$ where the coefficients of $L$ depend $C^{\8}$-smoothly on $Du_n$, and $f_n$ depends $C^k$-smoothly on $Du_n$. By condition b) above, we have $u_n\in C^{1,\alfa}(D_{\delta})$ for all $n$. Thus, all these coefficients are bounded in the $C^{0,\alfa}(D_{\delta})$ norm.
Then, by the usual Schauder estimates (see Gilbarg-Trudinger, \cite{GT} Chapter 6), for any $\delta' \in (0,\delta)$ we conclude that there exists a constant $C'$ (again independent of $n$) such that $||u_n||\leq C'$ in the $C^{2,\alfa}(D_{\delta'})$ norm. In particular, all coefficients of $L[u_n]=f_n$ are uniformly bounded in the $C^{1,\alfa}(D_{\delta'})$ norm. By repeating this argument we eventually obtain

$$||u_n||_{C^{k+2,\alfa}(D_{\delta'})} \leq C'', \hspace{1cm} 0<\alfa <1,$$ for some constant $C''$ independent of $n$. In these conditions, we may apply the Arzela-Ascoli theorem, and deduce by (ii), (iv) that a subsequence of the functions $u_n$ converge on $D_{\delta'}$ in the $C^{k+2}$ topology to a solution $u\in C^{k+2}(D_{\delta'})$ to
  \begin{equation}\label{eqhn2}
 {\rm div}\left(\frac{Du}{\sqrt{1+|Du|^2}}\right) = 2 \cH (Z_u), \hspace{1cm} Z_u:= \frac{(-Du,1)}{\sqrt{1+|Du|^2}}.
 \end{equation}
Thus, the graph of $u$ is an $\cH$-surface in $\R^3$ that by construction passes through $p$, and has second fundamental form bounded by $C$.
 
Consider next some $y\in D_{\delta'}$, and let $q$ be the corresponding point in the graph of $u$. It is then clear that there exist points $q_n\in \R^n$ in the graphs of $u_n$, all corresponding to $y$, and such that $q_n\to q$. Thus, passing to a subsequence if necessary so that condition (i) is fulfilled, we can repeat the same process above, this time with respect to the points $q_n$ and $q$. In this way, we obtain an $\cH$-surface $\Sigma$ in $\R^3$ that extends the graph of $u$ over $D_{\delta'}$.

Again by (i), it follows by a standard diagonal process that $\Sigma$ can be extended to a complete $\cH$-surface (which will also be denoted by $\Sigma$), that passes through $p$, and whose second fundamental form is bounded by $C$. Moreover, $\Sigma$ is by construction a limit in the $C^{k+1}$ topology on compact sets of the sequence of surfaces $(\Sigma_n)_n$; note that some other limit connected components could also appear in this process. This completes the proof.
\end{proof}

\section{Properly embedded $\cH$-surfaces}\label{sec:structure}
%
%In this section we will let $\cH$ be a positive $C^2$ function on $\S^2$ satisfying the following symmetry conditions:
%\begin{equation}\label{simcon}
%\def\arraystretch{1.3} \begin{array}{lll}\cH(x_1,x_2,x_3)=\cH(-x_1,x_2,x_3),  \\ \cH(x_1,x_2,x_3)=\cH(x_1,-x_2,x_3), \\ \cH(x_1,x_2,x_3)=\cH(x_1,x_2,-x_3).
% \end{array}
%\end{equation}
%In particular, $\cH(x)=\cH(-x)>0$ for all $x\in \S^2$, which implies that the Guan-Guan $\cH$-sphere $S_{\cH}$ exists. By uniqueness, $S_{\cH}$ is symmetric with respect to the three coordinate planes.
%

In this section we will study properly embedded $\cH$-surfaces of finite topology in $\R^3$, for $\cH\in C^1(\S^2)$, $\cH>0$. In Section \ref{diacur} we will recall a diameter estimate for horizontal sections of graphs with positive mean curvature by Meeks (Lemma \ref{lemi}), and we will obtain a curvature estimate for $\cH$-surfaces away from their boundary, see Theorem \ref{th:curv} and Remark \ref{rem:curv}. In Section \ref{hees} we will provide several a priori height estimates for $\cH$-graphs with zero boundary values over closed, not necessarily bounded, planar domains. These estimates will be used in Section \ref{estru} to study properly embedded $\cH$-surfaces in $\R^3$.

\subsection{Curvature and horizontal diameter estimates}\label{diacur}

The next result is essentially due to Meeks \cite{Me}:

\begin{lem}\label{lemi}
Let $\Sigma\subset \R^3$ be a graph $z=u(x,y)$ over a closed (not necessarily bounded) domain of $\R^2$, with zero boundary values. Assume that the mean curvature $H_{\Sigma}$ of $\Sigma$ satisfies $H_{\Sigma}>H_0$ for some $H_0>0$.

Then, for every $t>2/H_0$, the diameter of each connected component of $\Sigma\cap \{|z|=t\}$ is at most $2/H_0$. In particular, all connected components of $\Sigma\cap \{|z|\geq t\}$ for $t>2/H_0$ are compact.
\end{lem}
\begin{proof}
The argument follows the ideas of the proof of Lemma 2.4 in \cite{Me}, so we will only give here a sketch of it, following a slightly simplified version of Meeks' original proof, that can be found in \cite[Theorem 4]{AEG}, or \cite[Theorem 6.2]{EGR}.

Without loss of generality, we may assume that $u\geq 0$ and that, if $\cU$ denotes the unique connected component of $\R^3$ determined by the plane $z=0$ and the graph $\Sigma$, the mean curvature vector of $\Sigma$ points towards $\cU$.

Assume that there exist $t>2/H_0$ and a connected component of $\Sigma\cap\{z=t\}$ with a diameter greater than $2/H_0$. Let $\Omega\subset \R^2$ denote the domain where the graph $\Sigma$ is defined. Then, there exists a simple arc $\Gamma\subset\Omega$ such that the Euclidean distance between its extrema $p_1,p_2$ is greater than $2/H_0$, and so that $u(p)\geq t$ for all $p\in \Gamma$. Besides, there is no restriction in assuming that the Euclidean distance between any other two points of $\Gamma$ is smaller than the distance from $p_1$ to $p_2$. Up to a horizontal isometry, we can take $p_1=(-x_0,0),p_2=(x_0,0)$, with $x_0>1/H_0$.

In this way, the ``rectangle'' surface with boundary $S=\Gamma\times[0,t]$ lies entirely in $\cU$. Besides, $S$ divides the solid region
$$
{\cal C}=\{(x,y,z)\in\R^3:\ |x|\leq x_0,\ 0\leq z\leq t\}
$$
into two connected components ${\cal C}_1,{\cal C}_2$.

Hence, we can place a sphere of radius $1/H_0$ inside ${\cal C}_1$, and move it continuously towards ${\cal C}_2$ without leaving the interior of ${\cal C}$. Consider now just the piece of sphere that passes through $S$ into ${\cal C}_2$. It its clear that this piece cannot touch $\Sigma$, by the mean curvature comparison principle. Hence, the sphere could go through $S$ completely, and end up being contained in  ${\cal C}_2\cap{\cal U}$. But, obviously, a sphere of radius $1/H_0$ cannot be contained in the connected component $\cU$, since we could then move it upwards until reaching a first contact point with $\Sigma$, and this would contradict again the mean curvature comparison principle.
\end{proof}
The next result is a curvature estimate inspired by \cite{RST}.

\begin{teo}\label{th:curv}
Let $\Lambda,d,\rho$ be positive constants. Then there exists $C=C(\Lambda,d,\rho)>0$ such that the following assertion is true:

Let $\Sigma$ be any immersed oriented surface in $\R^3$, possibly with non-empty boundary, let $\sigma, H_{\Sigma},\eta$ denote, respectively, its second fundamental form, its mean curvature and its unit normal, and assume that:
 \begin{equation}\label{asscur1}
|H_{\Sigma}| + | \nabla H_{\Sigma} | \leq \Lambda \hspace{1cm} \text{on $\Sigma$}.
 \end{equation}
 \begin{equation}\label{asscur2}
\eta(\Sigma)\subset \S^2 \text{ omits a spherical disk of radius $\rho$.}
 \end{equation}
Then, for any $p\in \Sigma_d :=\{q\in \Sigma: d_{\Sigma}(q,\parc \Sigma) \geq d\}$, we have $$|\sigma(p)|\leq C.$$
\end{teo}
\begin{proof}
Arguing by contradiction, assume that the statement is not true, i.e. there is a sequence $f_n:\Sigma_n \flecha \R^3$ of immersed oriented surfaces in $\R^3$ satisfying \eqref{asscur1}, \eqref{asscur2}, and points $p_n\in \Sigma_n$ such that $d_{\Sigma_n}(p_n,\parc \Sigma_n)\geq d$ and $|\sigma_n(p_n)|>n$ for all $n$, where $\sigma_n$ is the second fundamental form of $f_n$. Note that by rotating each $f_n(\Sigma_n)$ adequately in $\R^3$, we may assume that the Gaussian images of \emph{all} the $f_n:\Sigma_n\flecha \R^3$ omit the open spherical disk of radius $\rho$ centered at the north pole of $\S^2$.

Consider the compact intrinsic metric disk $D_n=B_{\Sigma_n}(p_n,d/2)$ in $\Sigma_n$, which by construction is at a positive distance from $\parc \Sigma_n$. Let $q_n$ be the maximum on $D_n$ of the function 
$$
h_n(x)=|\sigma_n(x)|d_{\Sigma_n}(x,\partial D_n)
$$
Clearly, $q_n$ lies in the interior of $D_n$, as $h_n$ vanishes on $\parc D_n$. Let $\lambda_n=|\sigma_n(q_n)|$ and $r_n=d_{\Sigma_n}(q_n,\partial D_n)$. Then,
\begin{equation}\label{lann}
\lambda_n r_n=|\sigma_n(q_n)| d_{\Sigma_n}(q_n,\partial D_n)=h_n(q_n)\geq h_n(p_n)>\frac{d\,n}{2}.
\end{equation}
Note that this implies that $(\landa_n)_n\to \8$ as $n \to \8$. Also, note that for every $z_n\in B_{\Sigma_n}(q_n,r_n/2)$ we have
\begin{equation}\label{destr}
d_{\Sigma_n}(q_n,\partial D_n)\leq 2 d_{\Sigma_n}(z_n,\partial D_n).
\end{equation}

Consider now the immersed oriented surfaces $g_n:B_{\Sigma_n}(q_n,r_n/2)\flecha \R^3$ obtained by applying a rescaling of factor $\landa_n$ to the restriction of $f_n$ to $B_{\Sigma_n}(q_n,r_n/2)$; that is, $g_n=\landa_n f_n$ restricted to $B_{\Sigma_n}(q_n,r_n/2)$. For short, we will sometimes write $M_n$ to denote this immersed surface given by $g_n$. % in short, we will write $M_n=\lambda_n B_{\Sigma_n}(q_n,r_n/2)$. 

By \eqref{destr}, we have the following estimate for the second fundamental form $\hat{\sigma}_n$ of $M_n$ at any point $z_n$ in $B_{\Sigma_n}(q_n,r_n/2)$:
\begin{equation}\label{unisec}
|\hat{\sigma}_{n} (z_n)|= \frac{|\sigma_n(z_n)|}{\lambda_n} =\frac{h_n(z_n)}{\lambda_n d_{\Sigma_n}(z_n,\partial D_n)}\leq\frac{h_n(q_n)}{\lambda_n d_{\Sigma_n}(z_n,\partial D_n)}=\frac{d_{\Sigma_n}(q_n,\partial D_n)}{d_{\Sigma_n}(z_n,\partial D_n)}\leq 2.\end{equation}
In particular, the norms of the second fundamental forms of the surfaces $M_n$ are uniformly bounded. Also note that, by construction, $|\hat{\sigma}_n(q_n)|=1$. By \eqref{lann}, the radii of $M_n$ diverge to infinity (recall that the \emph{radius} of a compact Riemannian surface with boundary is the maximum distance of points in the surface to its boundary).

Let now $\widetilde{M}_n$ denote the translation of $M_n$ that takes the point $g_n(q_n)$ to the origin of $\R^3$, and let $\xi_n\in \S^2$ denote the Gauss map image of $M_n$ at $q_n$. After passing to a subsequence, we may assume that $(\xi_n)_n\to \xi$ as $n\to \8$, for some $\xi\in \S^2$. By construction, the norm of the second fundamental form of $\widetilde{M_n}$ is at most $2$, and it is equal to $1$ at the origin.

We use next an argument similar to the one in the proof of Theorem \ref{compa} to show that a subsequence of the surfaces $\widetilde{M}_n$ converges uniformly on compact sets to a complete minimal surface $M_{\8}$. 

First, Proposition 2.3 in \cite{RST} ensures that there exist positive constants $\delta_0,\mu$ (independent of $n$) such that, for any $n$ large enough, we can view a neighborhood of the origin in $\widetilde{M_n}$ as a graph of a function $u_n$ over a disk $D^0_n$ of radius $\delta_0$ of its tangent plane $T_0 \widetilde{M}_n= \xi_n^{\perp}$, and such that $||u_n||_{C^2(D^0_n)} \leq \mu$. Since the vectors $\xi_n$ converge to $\xi$ in $\S^2$, after making if necessary $\delta_0$ (resp. $\mu$) smaller (resp. larger), and for every $n$ large enough, we have that the same properties hold with respect to the $\xi$ direction; that is:

 \begin{enumerate}
 \item[a)]
An open neighborhood of the origin in $\widetilde{M}_n$ is the graph $x_3=u_n(x_1,x_2)$ of a function $u_n$ over the Euclidean disk $\mathcal{D}_0:=D(0,\delta_0)$ of radius $\delta_0$ in $\Pi_0=\xi^{\perp}$; here $(x_1,x_2,x_3)$ are orthonormal Euclidean coordinates centered at the origin, with $\frac{\parc}{\parc x_3}=\xi$.
 \item[b)]
The $C^2$ norm of $u_n$ in $\mathcal{D}_0$ is at most $\mu$.
 \end{enumerate}

Let $H_n(x_1,x_2)$ denote the mean curvature function of $\widetilde{M}_n$ in these coordinates. Note that, by \eqref{asscur1} and the fact that the factors $\landa_n$ diverge to $\8$, the functions $H_n$ are uniformly bounded in the $C^1(\mathcal{D}_0)$ norm, and as a matter of fact they converge uniformly to zero in that norm. Also note that, since the graph of $u_n$ has mean curvature $H_n$, then $u_n$ is a solution to the linear elliptic PDE for $u$
\begin{equation}\label{linpde} a_{11}(Du_n) u_{11} + 2 a_{12}(Du_n) u_{12} + a_{22} (Du_n) u_{22} = 2 H_n (1+|Du_n|^2)^{3/2},\end{equation} where $u_{ij}$ denotes second derivatives of $u$ with respect to the variables $x_i,x_j$, and the coefficients $a_{ij}$ are smooth functions. As, by condition b) above, the functions $u_n$ are uniformly bounded in the $C^{1,\alfa}$ norm in $\mathcal{D}_0$, we conclude that all coefficients of \eqref{linpde} are bounded in the $C^{0,\alfa}(\mathcal{D}_0)$ norm. By Schauder theory, the $C^{2,\alfa}$-norms in any $D(0,\delta)\subset\subset \mathcal{D}_0$ of the functions $u_n$ are uniformly bounded.

Once here, we may repeat the last part of the proof in Theorem \ref{compa} using the Arzela-Ascoli theorem and a diagonal argument, and conclude that a subsequence of the surfaces $\widetilde{M}_n$ converges uniformly on compact sets in the $C^2$ topology to a complete minimal surface $M_{\8}$ of bounded curvature that passes through the origin (note that $M_{\8}$ is minimal since the mean curvatures of $\widetilde{M}_n$ converge by construction to zero). Moreover, the norm of the second fundamental form of $M_{\8}$ at the origin is equal to $1$.

Also, since all the surfaces $\widetilde{M}_n$ have been obtained by translations and homotheties in $\R^3$ of the original immersions $f_n:\Sigma_n\flecha \R^3$, and since all the Gauss map images of the $f_n$ omit an open spherical disk of radius $\rho$ of the north pole in $\S^2$, it follows that $M_{\8}$ also omits such an open disk. By a classical result of Osserman, according to which the Gauss map image of a complete non-planar minimal surface in $\R^3$ is dense in $\S^2$, we deduce that $M_{\8}$ is a plane. This contradicts the fact that the norm of the second fundamental form of $M_{\8}$ at the origin is equal to $1$. This contradiction proves Theorem \ref{th:curv}.
\end{proof}
\begin{remark}\label{rem:curv}
It is clear from the proof that, in Theorem \ref{th:curv}, one can remove Assumption \eqref{asscur1} and ask instead that $\Sigma$ is an $\cH$-surface for some fixed, prescribed, $\cH\in C^1(\S^2)$. In that case, the constant $C$ only depends on $d,\rho$ and the $C^1$ norm of $\cH$ in $\S^2$.
\end{remark}

It is interesting to compare Theorem \ref{th:curv} with the family of catenoids $C_{\ep}$ in $\R^3$, where $\ep>0$ is the necksize. When $\ep\to 0$, the curvature of $C_{\ep}$ blows up at its \emph{waist}. Moreover, if we consider, for $d_0>0$ fixed, the piece $C_{\ep}(d_0)$ of $C_{\ep}$ of all points that are at a distance less than $d_0$ from the waist, then the Gaussian image in $\S^2$ of $C_{\ep}(d_0)$ converges as $\ep \to 0$ to $\S^2$ minus two antipodal points. This shows that condition \eqref{asscur2} cannot be avoided in Theorem \ref{th:curv}.

\subsection{Height estimates for $\cH$-graphs}\label{hees}

In Definition \ref{uhai}, $\parc \Sigma$ is not necessarily bounded, and $\Sigma$ is not compact in general.

\begin{defi}\label{uhai}
Let $\cH\in C^1(\S^2)$, and choose some $v\in \S^2$. We will say that there exists a \emph{uniform height estimate for $\cH$-graphs in the $v$-direction} if there exists a constant $C=C(\cH,v)>0$ such that the following assertion is true:

For any graph $\Sigma$ in $\R^3$ of prescribed mean curvature $\cH$ oriented towards $v$ (i.e. $\esiz \eta,v\esde>0$ on $\Sigma$ where $\eta$ is the unit normal of $\Sigma$), and with $\parc \Sigma$ contained in the plane $\Pi=v^{\perp}$, it holds that the height of any $p\in \Sigma$ over $\Pi$ is at most $C$. \end{defi}

Clearly, minimal graphs in $\R^3$ do not have a uniform height estimate (e.g., half-catenoids are counterexamples). If $\cH$ is a positive constant, Meeks showed in \cite{Me} that $\cH$-graphs admit uniform height estimates. However, for a general $\cH\in C^1(\S^2)$ the situation is more complicated. For instance, there exist complete, strictly convex, rotational $\cH$-graphs converging to a cylinder for adequate rotationally symmetric positive functions $\cH\in C^1(\S^2)$; see \cite{BGM}. The existence of such graphs shows that there are no uniform height estimates for arbitrary choices of $\cH\in C^1(\S^2)$, $\cH>0$. We should also point out that Meeks' proof uses that the CMC equation is invariant by reflections with respect to tilted Euclidean planes, and this is not the case anymore for a general $\cH\in C^1(\S^2)$, not even in the rotationally symmetric case. Thus, our approach to provide uniform height estimates for $\cH$-graphs relies on different ideas.

%Our analysis in Section \ref{rotiper} also indicates that it is natural to assume $\cH>0$ in order to obtain a uniform height estimate for $\cH$-graphs, since in the case that $\cH$ is rotationally symmetric and vanishes somewhere, there exist entire, strictly convex $\cH$-graphs (and thus, there is no uniform height estimate); see Proposition \ref{hbowls}. Also note that if $\cH>0$, the condition $\esiz \eta,v\esde>0$ implies by the comparison principle that ${\rm int}(\Sigma)\subset \Pi^{-}$, where $\Pi^-$ is the connected component of $\R^3-\Pi$ that contains $-v$.

Let us fix some notation. In the next theorem we will assume after choosing new coordinates $(x_1,x_2,x_3)$ that $v=e_3$. We will denote $\S_+^2 =\S^2 \cap \{x_3> 0\}$,  $S^1 =\S^2\cap \{x_3=0\}$, and $\overline{\S_+^2}=\S_+^2\cup S^1$.

Let $\cH\in C^1(\overline{\S_+^2})$, $\cH>0$. By an \emph{$\cH$-hemisphere} in the $e_3$-direction we will mean a compact, strictly convex $\cH$-surface $\Sigma_{\cH}$ with boundary, such that ${\rm int}(\Sigma_{\cH})$ is an upwards-oriented graph $x_3=u(x_1,x_2)$ over a $C^2$ regular convex disk in $\R^2$, and whose Gauss map image is $\eta(\Sigma_{\cH})=\overline{\S_+^2}$.

%Given $\cH\in C^1(\overline{\S_+^2}))$, we will denote 
%\begin{equation}\label{hcom}\mathfrak{h} (x):= {\rm max} \{ \cH(\xi) : \esiz \xi,e_3\esde = x\} :[0,1]\flecha \R.
%\end{equation}
%Note that $\mathfrak{h}$ is continuous in $[0,1]$.

Given $\cH\in C^1(\overline{\S_+^2})$, recall that a necessary and sufficient condition for the existence of a closed curve $\gamma\subset \R^2$ such that the cylinder $\gamma \times \R\subset \R^3$ is an $\cH$-surface (see Corollary \ref{difci}) is that $\cH(\xi) \neq 0$ for every $\xi \in S^1$, and 
 \begin{equation}\label{conci}\int_{S^1} \frac{\xi}{\cH(\xi)} d\xi =0.
  \end{equation}
In that case, $\gamma$ is strictly convex and bounds a compact domain in $\R^2$, that we will denote by $\Omega_{\cH}$.

\begin{teo}\label{th:hees}
Given $\cH\in C^1(\overline{\S_+^2})$, $\cH>0$, \emph{any} of the following conditions on $\cH$ imply that there exists a uniform height estimate for $\cH$-graphs in the $e_3$-direction:
\begin{enumerate}
\item
Condition \eqref{conci} does \emph{not} hold.
 \item
Condition \eqref{conci} holds, and there exists a graph $\Sigma_0$ in the $e_3$-direction, oriented towards $e_3$, over a domain $\Omega\subset \R^2$ that contains $\Omega_{\cH}$, and with the property that $H_{\Sigma_0}(p)>\cH(\eta(p))$ for all $p\in \Sigma_0\cap (\Omega_{\cH}\times \R)$.
\item
There exists an $\cH$-hemisphere in the $e_3$-direction.
 \item
There is some $\cH^*\in C^1(\overline{\S_+^2})$, with $\cH=\cH^*$ in $S^1$ and $\cH^*>\cH$ in $\S_+^2$, for which there exists an $\cH^*$-hemisphere in the $e_3$-direction.
%  \item
%There exist $\mathfrak{h}_* \in C^2([0,1])$ such that, for the function $\mathfrak{h}$ given by \eqref{hcom}, we have $\mathfrak{h}_*(x)>\mathfrak{h}(x)$ for all $x\in (0,1)$, $\mathfrak{h}_*(0)=\mathfrak{h}(0)$ and $\mathfrak{h}_*'(0) \leq 0$.
%\item
%The function $\mathfrak{h}$ in \eqref{hcom} is $C^2$ at $x=0$, with
 %\begin{equation}\label{conrota}
%\mathfrak{h}'(0)\leq 0.
 %\end{equation}
 \item
${\rm max} \, \cH  <  2\, {\rm min} \, \cH|_{S^1}.$
\end{enumerate}

\end{teo}

We will prove Theorem \ref{th:hees} in Section \ref{profhes}, and devote the rest of the present Section \ref{hees} to discuss the sufficient conditions described in Theorem \ref{th:hees}, and to deduce some corollaries from it. 

The first condition in Theorem \ref{th:hees} indicates that for generic, non-symmetric choices of $\cH\in C^1(\overline{\S_+^2})$, $\cH>0$, there exist uniform height estimates for $\cH$-graphs. The second condition gives, for the remaining cases of $\cH$, a very general sufficiency property for the existence of uniform height estimates for $\cH$-graphs, in terms of the existence of an adequate \emph{barrier} $\Sigma_0$. The rest of sufficient conditions in Theorem \ref{th:hees} are obtained by applying the second condition to situations where we can construct the barrier $\Sigma_0$.

One particular consequence of Theorem \ref{th:hees} of special interest is:

\begin{cor}\label{cor:esfe}
Let $\cH\in C^1(\S^2)$, $\cH>0$, and assume that there exists a strictly convex $\cH$-sphere $S_{\cH}$ in $\R^3$. Then, there exist uniform height estimates for $\cH$-graphs with respect to any direction $v\in \S^2$.
\end{cor}
\begin{proof}
Simply observe that $\Sigma_{\cH,v}:=\{p\in S_{\cH} : \esiz \eta_S(p),v\esde \geq 0\}$ is an $\cH$-hemisphere in the $v$-direction, and apply item 3 of Theorem \ref{th:hees}; here, $\eta_S$ is the unit normal of $S_{\cH}$.
\end{proof}

The next general result is immediate from item 5 of Theorem \ref{th:hees}.

\begin{cor}
Let $\cH\in C^1(\overline{\S_+^2})$, $\cH>0$. Choose any $H_0> {\rm max} \, \cH - 2 {\rm min} \, \cH |_{S^1}$. Then there exists a uniform height estimate for graphs of prescribed mean curvature $\cH+H_0$ in the $e_3$-direction.
\end{cor}

If we impose some additional symmetry to the function $\cH$, we can obtain from Theorem \ref{th:hees} more definite sufficient conditions for the existence of a height estimate. For instance:

\begin{cor}
Let $\cH\in C^2(\overline{\S_+^2})$, $\cH>0$, be rotationally symmetric, i.e. $\cH(x)= \mathfrak{h}(\esiz x,e_3\esde)$ for some $\mathfrak{h}\in C^2([0,1])$. Assume that $\mathfrak{h}'(0)\leq 0$.

Then, there is a uniform height estimate for $\cH$-graphs in the $e_3$-direction.
\end{cor}
\begin{proof}
Assume first that $\mathfrak{h}'(0)=0$. Then, we can extend $\mathfrak{h}$ to a positive, even, $C^2$ function on $[-1,1]$. Hence, $\cH$ can also be extended to a positive $C^2$ function on $\S^2$, also denoted $\cH$, so that $\cH(x)=\cH(-x)$ for all $x\in \S^2$. By Theorem \ref{tgg}, there exists a strictly convex $\cH$-sphere $S_{\cH}$. Thus, the result follows from Corollary \ref{cor:esfe}. %if we define $$\Sigma_{\cH}^+ :=\{p\in S_{\cH} : \esiz \eta_S(p),e_3\esde \geq 0\},$$ where $\eta_S$ is the inner unit normal of $S_{\cH}$, then $\Sigma_{\cH}^+$ is an $\cH$-hemisphere and we conclude by the third sufficient condition in Theorem \ref{th:hees}.

Assume now that $\mathfrak{h}'(0)<0$. Then, we can construct a function $h^*\in C^2([-1,1])$, with $h^*(t)=h^*(-t)$ for all $t$, and such that $h^*(0)=\mathfrak{h}(0)$ and $h^*(t)>\mathfrak{h}(t)$ for all $t\in (0,1]$. Let $\cH^*\in C^2(\S^2)$ be defined by $\cH^*(x):= h^*(\esiz x,e_3\esde)$, and note that $\cH^*(x)=\cH^*(-x)>0$ for all $x\in \S^2$. Arguing as above, there exists an $\cH^*$-hemisphere in the $e_3$-direction. Since it is clear that $\cH^*=\cH$ in $S^1$ and $\cH^*>\cH$ in $\S_+^2$, we concluded the desired result by the fourth sufficient condition in Theorem \ref{th:hees}.
\end{proof}
%\begin{remark}
%The condition $\mathfrak{h}'(0)\leq 0$ in the above corollary cannot be removed altogether. Indeed, as we explained after Definition \ref{uhai}, our analysis in Section \ref{sec:nocmc} shows that there exist rotationally symmetric functions $\cH\in C^{\8}(\S^2)$, $\cH>0$, that do not admit uniform height estimates.
%\end{remark}
For the case in which $\cH$ is invariant under the symmetry with respect to a geodesic of $\S^2$, we have an estimate for compact embedded $\cH$-surfaces, not necessarily graphs.

\begin{cor}\label{cor:hees}
Let $\cH\in C^1(\S^2)$, $\cH>0$, satisfy:
 \begin{enumerate}
 \item
 $\cH(x_1,x_2,x_3)=\cH(x_1,x_2,-x_3)$ for all $x=(x_1,x_2,x_3)\in \S^2$.
  \item
There exists a uniform height estimate for $\cH$-graphs in the directions of $e_3$ and $-e_3$.
 \end{enumerate}
Then, there is some constant $C=C(\cH)>0$ such that the following assertion is true: any compact embedded $\cH$-surface $\Sigma$ in $\R^3$, with $\parc \Sigma$ contained in the plane $x_3=0$, lies in the slab $|x_3|\leq C$ of $\R^3$.
%Assume that $\cH\in C^2(\S^2)$ is symmetric with respect to the geodesic $\S^2\cap \{x_3=0\}$, and that its associated function $\mathfrak{h}(x):={\rm max}\{\cH(\xi): \esiz \xi,e_3\esde =x\} \in C^0([-1,1])$ is $C^2$ at $x=0$.
\end{cor}
\begin{proof}
It is an immediate consequence of the Alexandrov reflection principle applied to the family of horizontal planes in $\R^3$. Note that, due to the symmetry condition imposed on $\cH$, we can apply this reflection method in that particular direction; see Lemma \ref{sime}.
%It is immediate from the properties imposed to $\cH$ that we are in the situation of the fourth sufficient condition in Theorem \ref{th:hees}, and that equation \eqref{conrota} is, in this case, $\mathfrak{h}'(0)=0$. This implies that we are also in the situation of applying the fourth condition in Theorem \ref{th:hees} to $\cH$-graphs in the $-e_3$ direction. In particular, this proves that Corollary \ref{cor:hees} holds for the case that $\Sigma$ is an $\cH$-graph $x_3=u(x_1,x_2)$, with any orientation, for some constant $C'=C'(\cH)$.
%
%Once here, a standard application of Alexandrov's reflection principle with respect to the family of horizontal planes proves that Corollary \ref{cor:hees} holds in full generality for $C:=2C'$; observe that we can apply the reflection principle due to the symmetry imposed on $\cH$.
\end{proof}

\subsection{Proof of Theorem \ref{th:hees}}\label{profhes}

We start by proving the first two items. Arguing by contradiction, assume that $\cH\in C^1(\overline{\S_+^2})$, $\cH>0$, is a function for which there is no uniform height estimate for $\cH$-graphs in the $e_3$-direction. %Our aim is to prove that $\cH$ cannot satisfy any of the conditions stated in the theorem.
So, there exists a sequence $(\Sigma_n)_n$ of $\cH$-graphs with respect to the $e_3$-direction, oriented towards $e_3$, and with boundary $\parc \Sigma$ contained in $\Pi=e_3^{\perp}=\{x_3=0\}$, and points $p_n\in \Sigma_n$ such that the height of $p_n$ over $\Pi$ is greater than $n$. Note that $\Sigma_n\subset \{x_3\leq 0\}$.

Take now $H_0\in (0,{\rm min}\, \cH)$, and denote $\Sigma_n^*:=\Sigma_n\cap \{|x_3| \geq 4/H_0\}$. By hypothesis, $p_n\in \Sigma_n^*$ for $n$ large enough. By Lemma \ref{lemi}, the connected component $\Sigma_n^0$ of $\Sigma_n^*$ that contains $p_n$ is compact, and contained inside a vertical solid cylinder of radius $2/H_0$. 

Let $q_n\in \Sigma_n^0$ be a point of maximum height of $\Sigma_n^0$, and let $\Sigma_n^1:=\Sigma_n^0-q_n$ denote the translation of $\Sigma_n^0$ that takes $q_n$ to the origin in $\R^3$. By Theorem \ref{th:curv} and Remark \ref{rem:curv}, the norms of the second fundamental form of the graphs $\Sigma_n^1$ are uniformly bounded by some positive constant $C>0$ that only depends on $H_0$ and $||\cH||_{C^1(\S^2)}$, and not on $n$. Moreover, the distances in $\R^3$ of the origin to $\parc \Sigma_n^1$ diverge to $\8$. By Theorem \ref{compa}, we deduce that, up to a subsequence, there are smooth compact sets $K_n$ of the graphs $\Sigma_n^1$, all of them containing the origin and with horizontal tangent plane at it, and that converge uniformly in the $C^2$ topology to a complete $\cH$-surface $\Sigma_{\8}$ of bounded curvature that passes through the origin.

Let us define next $\nu_{\8}:=\esiz \eta_{\8},e_3\esde$, where $\eta_{\8}$ is the unit normal of $\Sigma_{\8}$. Note that $\nu_{\8}({\bf 0})=1$. As all graphs $\Sigma_n^1$ are oriented towards $e_3$, we deduce that $\nu_{\8}\geq 0$ on $\Sigma_{\8}$.

Furthermore, it follows directly from Corollary \ref{nueq} (to be proved in Section \ref{sec:stable}) that $\nu_{\8}$ is a solution to a linear elliptic equation on $\Sigma_{\8}$ of the  form 
 \begin{equation}\label{saei}
 \Delta \nu_{\8} + \esiz X,\nabla \nu_{\8}\esde + q \nu_{\8}=0,
 \end{equation}
where $\Delta,\nabla$ denote the Laplacian and gradient operators on $\Sigma_{\8}$, $X\in \X(\Sigma_{\8})$, and $q\in C^2(\Sigma_{\8})$.  By the maximum principle for \eqref{saei}, and the condition $\nu_{\8}\geq 0$, we conclude that either $\nu_{\8} \equiv 0$ on $\Sigma_{\8}$ (which cannot happen since $\nu_{\8}({\bf 0})=1$), or $\nu_{\8}>0$ on $\Sigma_{\8}$. Therefore, $\Sigma_{\8}$ is a local vertical graph, i.e. for every $p\in \Sigma_{\8}$ it holds that $T_p \Sigma_{\8}$ is not a vertical plane in $\R^3$.

Once here, and since $\Sigma_{\8}$ is a limit of compact pieces of the graphs $\Sigma_n^1$, it is clear that $\Sigma_{\8}$ is itself a proper $\cH$-graph in $\R^3$ oriented towards $e_3$. By construction, this graph has horizontal tangent plane at the origin, it has bounded second fundamental form, and lies entirely in the closed half-space $\{x_3\geq 0\}$. Moreover, since each $\Sigma_n^0$ lies inside a vertical solid cylinder in $\R^3$ of radius $2/H_0$, we deduce that all points of $\Sigma_{\8}$ lie at a distance in $\R^3$ at most $4/H_0$ from the $x_3$-axis. %From now on we will consider the connected component of $\Sigma_{\8}$ that passes through the origin, and we will still denote this component by $\Sigma_{\8}$.

Since $\Sigma_{\8}$ is a complete proper graph, and at a distance at most $4/H_0$ from the $x_3$-axis, it is clear that for any $q_0\in \parc \Omega$ there exists a diverging sequence of points $(a_n)_n\in \Sigma_{\8}$, whose horizontal projections converge to $q_0$, and such that $\nu_{\8}(a_n)\to 0$.

Express now $\Sigma_{\8}$ as $x_3=f(x_1,x_2)$ over an open bounded domain $\Omega\subset\R^2$. Let $L$ denote a straight line in $\R^2$ far away from $\Omega$, and let us start moving it towards $\Omega$ until it reaches a first contact point $q_0\in \parc \Omega$ with the compact set $\overline{\Omega}$.

%, let $\partial_{ext}\Omega$ be the exterior boundary component of $\Omega$ (i.e. the component of $\parc \Omega$ that bounds all other components), and 

Choose a diverging sequence $(a_n)_n$ in $\Sigma_{\8}$ whose horizontal projections converge to $q_0 \in \partial \Omega$, and consider the vertical translations $\Sigma_\8^n=\Sigma_\8-(0,0,a_n^{3})$, where $a_n^3$ denotes the third coordinate of $a_n$. Up to a subsequence we can suppose that the unit normals of $\Sigma_{\8}$ at $a_n$ converge to a fixed vector $\eta_0\in \S^2$, which is horizontal. Again, a similar compactness argument to the ones above ensures that a subsequence of the graphs $\Sigma_\8^n$ converges to a complete $\cH$-surface $\Sigma_\8^*$, that passes through $q_0\in \parc \Omega$.

Note that, by construction, $\Sigma_{\8}^*$ is contained in $\overline{\Omega}\times \R$. Moreover, it is clear from the fact that the sequence $a_n^3$ diverges to $\8$ that $(\Omega\times\R)\cap\Sigma_\8^*=\emptyset$ (observe that $\Sigma_{\8}^*$ is constructed from divergent vertical translations of compact pieces of $\Sigma^*$). This implies that $\Sigma_\8^*$ is contained in $\partial\Omega\times\R$, and that the connected component of $\Sigma_\8^*$ that contains the point $(q_0,0)$ is contained in $(\partial_{0}\Omega)\times\R$, where $\partial_0 \Omega$ denotes the connected component of $\parc \Omega$ that contains $q_0$. If we keep calling this connected component as $\Sigma_{\8}^*$, then we have $\Sigma_{\8}^*=\alfa \times \R$, where $\alfa\subset \parc_{0}\Omega$ is a regular curve in $\R^2$ that verifies equation \eqref{planmin}, since $\Sigma_{\8}^*$ has prescribed mean curvature $\cH$. Note that $\alfa$ is a closed curve, since $\parc\Omega$ is compact, and so $\Sigma_{\8}^*$ is a complete flat cylinder with vertical rulings, diffeomorphic to $\S^1\times \R$, i.e. one of the surfaces in Corollary \ref{difci}. 

Since $\alfa\cap \Omega=\emptyset$, it follows by connectedness that $\Omega$ is contained in one of the two regions of $\R^2$ separated by $\alfa$. Moreover, from the way that the point $q_0\in \parc \Omega$ was chosen, it is clear that $\Omega$ is actually contained in the \emph{inner} region bounded by $\alfa$. To see this, one should observe that the $\Sigma_{\8}^n$ are graphs over $\Omega$, that they converge uniformly on compact sets to $\Sigma_{\8}^*=\alfa\times \R$, and that $\Omega$ does not intersect neither $\alfa$ nor $L$.

We are now in the conditions to prove the first two items stated in Theorem \ref{profhes}. It is important to recall here that our argument was by contradiction.

Since $\alfa$ is a closed curve, it follows from Corollary \ref{difci} that equation \eqref{conci} holds. This proves item 1 of Theorem \ref{profhes}.

We next prove item 2. Assume that there exists a graph $\Sigma_0$ in the conditions of that item; that is, the domain $\Omega_0\subset \R^2$ over which $\Sigma_0$ is defined contains $\alfa$, and $H_{\Sigma_0}(p) > \cH(\eta(p))$ for all $p\in \Sigma_0\cap (\Omega_{\cH} \times \R)$. In particular, $\overline{\Omega}\subset\Omega_0$, where $\Omega$ is the domain of the graph $\Sigma_{\8}$, which has prescribed mean curvature $\cH$. Recall that $\Sigma_{\8}$ is contained in $\{x_3\geq 0\}$. Hence, we can move $\Sigma_0$ downwards by vertical translations so that its restriction to the compact set $\overline{\Omega}\subset \Omega_0$ is contained in $\{x_3<0\}$, and then start moving it upwards until reaching a first contact point with $\Sigma_\8$. This contradicts the mean curvature comparison principle, since $H_{\Sigma_0} > \cH\circ \eta$ on $ \Sigma_0\cap (\Omega_{\cH} \times \R)$ by hypothesis. This contradiction proves item 2 of Theorem \ref{profhes}.

\begin{remark}\label{2item2}
The statement of item 2 also holds for the case that $\Sigma_0$ is an upwards-oriented $\cH$-graph in the $e_3$-direction, defined over a domain that contains $\Omega_{\cH}$ in its interior. The only difference in the proof with this new hypothesis is that the desired contradiction is reached by using the maximum principle of $\cH$-surfaces, and not the mean curvature comparison principle. \end{remark}

We next prove item 4, as an application of item 2. Let $\cH, \cH^*$ be in the conditions of item 4, and let $\Sigma_{\cH}^*$ denote an $\cH^*$-hemisphere in the $e_3$-direction. Note that $\Sigma_{\cH}^*$ can be seen as an upwards-oriented graph $x_3=u(x_1,x_2)$ over a closed strictly convex disk $\Omega_0$ with $C^2$ regular boundary $\Gamma=\parc \Omega_0$. By item 1, we may assume that condition \eqref{conci} holds, and in particular we can consider the closed curve $\gamma$ in $\R^2$ such that $\gamma\times \R$ is an $\cH$-surface in $\R^3$ (see the comments before the statement of Theorem \ref{th:hees}). If $\kappa_1^*,\kappa_2^*$ denote the (positive) principal curvatures of $\Sigma_{\cH}^*$, then at any $p\in \Sigma_{\cH}^*$ at which the unit normal $\eta^*(p)$ is a horizontal vector $\xi\in S^1$, we have (since $\cH=\cH^*$ in $S^1$) that
\begin{equation}\label{dit32}\kappa_1^*(p)+\kappa_2^*(p)= 2\cH^*(\xi) = 2\cH(\xi) = \kappa_{\gamma}(p'),\end{equation} where $p'$ is the point of $\gamma$ with unit normal equal to $\xi$, and $\kappa_{\gamma}$ denotes the (positive) geodesic curvature of $\gamma$. This implies that for any unit vector $v\in T_p \Sigma_{\cH}^*$, the second fundamental form $\sigma^*$  of $\Sigma_{\cH}^*$ satisfies 
 \begin{equation}\label{dit31} 0<\sigma^*_p (v,v)<\kappa_{\gamma}(p').\end{equation} 
 Since $\Sigma_{\cH}^*$ is an $\cH$-hemisphere, the points in $ \parc\Sigma_{\cH}^*=\{p\in \Sigma_{\cH}^* : \esiz \eta^*(p), e_3\esde =0\}$ project regularly onto the convex curve $\Gamma$. Hence, $\gamma,\Gamma$ are two closed, strictly convex planar curves that, by \eqref{dit31}, satisfy the following condition: if $n_{\Gamma}$ and $n_{\gamma}$ denote the inner unit normals of $\Gamma, \gamma$, then:
  \begin{equation}\label{comcur}
 \kappa_{\Gamma}(p)<\kappa_{\gamma}(p') \text{ whenever $n_{\Gamma}(p)=n_{\gamma}(p')$. }
  \end{equation}
We will need at this point the following classical property of convex curves in $\R^2$, whose proof we omit:

\vspace{0.1cm}

{\it Fact: If $\Gamma,\gamma$ are two closed, strictly convex regular planar curves that satisfy condition \eqref{comcur}, then $\gamma$ is contained in the interior region bounded by some translation of $\Gamma$.}

%\vspace{0.1cm}
%
%{\it Proof of the Assertion:} {\bf (Escribir bien)} After an adequate translation of $\Gamma$, we can assume that there is some $p\in \Gamma\cap \gamma$ with $n_{\Gamma}(p)=n_{\gamma}(p)$ (we are still denoting this translated curve by $\Gamma$). By \eqref{comcur}, $\gamma$ is contained in the convex region bounded by $\Gamma$. Arguing by contradiction, assume that there are points in $\Gamma\cap \gamma$ other than $p$. By convexity, it is possible to choose such a point $p^*\in \Gamma\cap \gamma$ such that the following property holds: there exist arcs $\Gamma_0, \gamma_0$ of $\Gamma,\gamma$ with endpoints $p,p^*$, such that both $\Gamma_0,\gamma_0$ are graphs with respect to the same unitary direction $\xi$ of $\R^2$. Assume for definiteness that $\Gamma_0$ lies \emph{below} $\gamma_0$ with respect to $\xi$, i.e. that $\xi$ points towards the interior region of $\Gamma$. Once here, let us move $\Gamma_0$ upwards in the direction of $\xi$, i.e. consider $\Gamma_0^t:= \Gamma_0 + t \xi$, until we reach a last contact point of the family $\{\Gamma_0^t\}_{t\geq 0}$ with $\gamma_0$. Let $t_0$ denote the instant where this last contact points happens. Note that this last contact point takes place at an interior point of $\Gamma_0^{t_0}\cup \gamma_0$, and so the unit normal of both curves agree at this point. As the curvature of $\Gamma_0^t$ does not depend on the value $t$, we contradict \eqref{comcur} since $\Gamma_0^{t_0}$ lies \emph{above} $\gamma_0$ (with respect to $\xi$) around the contact point. This proves the assertion.

\vspace{0.1cm}

It follows then from the Fact above that, up to a translation of $\Sigma_{\cH}^*$, the convex disk $\Omega_0$ of $\R^2$ over which $\Sigma_{\cH}^*$ is a graph contains $\gamma=\parc \Omega_{\cH}$ in its interior. Hence, we can conclude directly the existence of a uniform height estimate for $\cH$-graphs by applying the already proved item 2 of the theorem. This proves item 4.

The proof of item 3 is analogous to the one of item 4, using the alternative formulation of item 2 that is explained in Remark \ref{2item2}, instead of item 2 itself. We omit the details. 
%Specifically, we only need to note that item 2 is also true if we assume  except that in this case we use the maximum principle for $\cH$-surfaces rather than the mean curvature comparison principle to obtain a contradiction.

%We next prove item 3, as an application of item 2. First, note that if $\mathfrak{h}_*(x)\in C^2([0,1])$ is positive and satisfies $\mathfrak{h}_*'(0)=0$, we can extend it to an even $C^2$ function on $[-1,1]$ by $\mathfrak{h}_*(x)=\mathfrak{h}_*(-x)$, which we will still call $\mathfrak{h}_*$. By Theorem \ref{delaunay}, or more in general by \cite{GG}, there exists a strictly convex rotational sphere $S_{\mathfrak{h_*}}$ in $\R^3$ of prescribed mean curvature $\mathfrak{h}_*$. 

%Assume now that the function $\mathfrak{h}_*$ in the above conditions also satisfies, with respect to the function $\mathfrak{h}$ defined in \eqref{hcom} in terms of $\cH$, that $\mathfrak{h}_*(0)=\mathfrak{h}(0)$ and $\mathfrak{h}_*(x)>\mathfrak{h}(x)$ for all $x\in (0,1)$. Let $\Sigma_0$ denote the \emph{lower hemisphere} of $S_{\mathfrak{h_*}}$, which is a strictly convex graph of prescribed mean curvature $\mathfrak{h}_*$, oriented towards $e_3$. Let $\kappa_1^0,\kappa_2^0$ denote the principal curvatures of $\Sigma_0$, and note that both of them are positive. So, at points where the unit normal of $\Sigma_0$ is horizontal, we have by the definition of $\mathfrak{h}$
%$$\kappa_1^0 + \kappa_2^0 = 2\mathfrak{h}_* (0) = 2\mathfrak{h}(0) \geq  $$

We next prove item 5. First, observe again that, by item 1, we may assume that condition \eqref{conci} holds, and so, we can consider again the closed planar curve $\gamma$ for which $\gamma\times \R$ is an $\cH$-surface in $\R^3$. Denote $H_0:={\rm max} \, \cH$, and let $\S^2(1/H_0)$ be the round sphere of constant mean curvature $H_0$. From the condition on $\cH$ in item 5, we get $2\cH(\xi)>H_0$ for all $\xi\in S^1$. This implies that the geodesic curvature of $\gamma$ satisfies $\kappa_{\gamma}>H_0$ at every point. So, this means that, up to a translation, $\gamma$ is contained in the open disk $D(0,1/H_0)$ of $\R^2$. Once here, we conclude the proof by applying item 2 to the lower hemisphere $\Sigma_0$ of $\S^2(1/H_0)$.
This finishes the proof of Theorem \ref{th:hees}.

\subsection{Properly embedded $\cH$-surfaces with one end}\label{estru}

The following result is due to Meeks \cite{Me}, and will play a key role in this section. See also Lemma 1.5 of \cite{KKS}.

%(Explicar qué significa anillo propiamente embebido).

\begin{teo}[\textbf{Plane separation lemma}]\label{seplem}
Let $\Sigma$ be a surface with boundary in $\R^3$, diffeomorphic to the punctured closed disk $\overline{\D}-\{0\}$. Assume that $\Sigma$ is properly embedded, and that its mean curvature $H_{\Sigma}$ satisfies $H_{\Sigma}(p)\geq H_0>0$ for every $p\in \Sigma$, and for some $H_0$.

Let $P_1,P_2$ be two parallel planes in $\R^3$ at a distance greater than $2/H_0$, and let $P_+,P_-$ be the connected components of $\R^3-[P_1,P_2]$, where $[P_1,P_2]$ is the open slab between both planes. Then, all the connected components of either $\Sigma\cap P^+$ or $\Sigma\cap P^-$ are compact.
\end{teo}

In what follows, we say that a surface $\Sigma$ has \emph{finite topology} if it is diffeomorphic to a compact surface (without boundary) with a finite number of points removed. If $\Sigma$ is properly embedded, each of such removed points corresponds then to an \emph{end} of the surface. Any such end is of \emph{annular type}, that is, the surface can be seen in a neighborhood of such punctures as a proper embedding of the punctured closed disk $\overline{\D}-\{0\}$ into $\R^3$, and thus is in the conditions of Theorem \ref{seplem}.

In the next theorem one should recall that, by our analysis in Section \ref{hees}, the existence of uniform height estimates for $\cH$-graphs holds in any direction $v\in \S^2$ for generic choices of $\cH\in C^1(\S^2)$, $\cH>0$. Thus, the second hypothesis in the theorem below, although a necessary one, is relatively weak.

\begin{teo}\label{t1pe}
Let $\cH\in C^1(\S^2)$, $\cH>0$, satisfy:
 \begin{enumerate}
 \item[(a)]
$\cH$ is invariant under the reflection in $\S^2$ that fixes a geodesic $\S^2\cap v^{\perp}$, for some $v\in \S^2$.
  \item[(b)]
There exists a uniform height estimate for $\cH$-graphs in the directions of $v$ and $-v$.
 \end{enumerate}
Then, there is some constant $D=D(\cH,v)>0$ such that the following assertion is true: any properly embedded $\cH$-surface in $\R^3$ of finite topology and one end is contained in a slab of width at most $D$ between two planes parallel to $\Pi=v^{\perp}$.
%
%compact embedded $\cH$-surface $\Sigma$ in $\R^3$, with $\parc \Sigma$ contained in the plane $x_3=0$, lies in the slab $|x_3|\leq C$ of $\R^3$.
%Assume that $\cH\in C^2(\S^2)$ is symmetric with respect to the geodesic $\S^2\cap \{x_3=0\}$, and that its associated function $\mathfrak{h}(x):={\rm max}\{\cH(\xi): \esiz \xi,e_3\esde =x\} \in C^0([-1,1])$ is $C^2$ at $x=0$.
%\end{cor}
\end{teo}
\begin{proof}
%Let $\Pi_1,\Pi_2$ be two planes parallel to $\Pi=v^{\perp}$, at a distance greater than $2/H_0$, where $H_0:={\rm min} \, \cH$. Since $\Sigma$ is properly embedded and only has one end, we can write $\Sigma=\Sigma_0 \cup \cA$, where $\Sigma_0$ is a compact surface with boundary, and $\cA$ is a proper embedding of $\overline{\D}-\{0\}$ into $\R^3$. Using this decomposition and the plane separation Lemma \ref{seplem}, we deduce that either $\Sigma\cap P_+$ or $\Sigma\cap P_-$ only has compact connected components. Say, for definiteness that $\Sigma\cap P_+$ has this property, where $\parc P_+ =\Pi_1$. By Corollary \ref{cor:hees}, $\Sigma\cap P_+$ is contained between $\Pi_1$ and its parallel plane $\Pi_1^*$ in $P_+$ at a distance $C$, where $C=C(\cH)$ is the constant in Corollary \ref{cor:hees}. In particular, $\Sigma$ is contained in the half-space determined by $\Pi_1^*$ that contains $\Pi_2$.
%
%****
Let $\Pi_1,\Pi_2$ be two planes parallel to $\Pi=v^{\perp}$, at a distance $2d$ greater than $2/H_0$, where $H_0:={\rm min} \, \cH$, and assume that both $\Pi_1,\Pi_2$ intersect $\Sigma$ (if such a pair of planes does not exist, then $\Sigma$ lies in a slab of $\R^3$ of width $2d$, and the result follows). After a change of Euclidean coordinates, we may assume that $v=e_3$, that $\Pi_1=\{x_3=d\}$, and $\Pi_2=\{x_3=-d\}$ for some $d>1/H_0$. Since $\Sigma$ is properly embedded and only has one end, we can write $\Sigma=\Sigma_0 \cup \cA$, where $\Sigma_0$ is a compact surface with boundary, and $\cA$ is a proper embedding of $\overline{\D}-\{0\}$ into $\R^3$. Using this decomposition and the plane separation lemma (Theorem \ref{seplem}), we deduce that either $\Sigma\cap \{x_3\geq d\}$ or $\Sigma\cap \{x_3\leq -d\}$ only has compact connected components. Say, for definiteness that $\Sigma\cap \{x_3\geq d\}$ has this property. By Corollary \ref{cor:hees}, $\Sigma\cap \{x_3\geq d\}$ is contained in the slab $d\leq x_3 \leq d+C$, where $C=C(\cH)$ is the constant appearing in that corollary. In particular, $\Sigma$ is contained in the half-space $\{x_3\leq d+C\}$.

Consider next the planes $x_3=d-2C$ and $x_3=-d-2C$. By the same arguments, at least one of $\Sigma\cap \{x_3\geq d-2C\}$ or $\Sigma\cap \{x_3\leq -d-2C\}$ only has compact connected components. In case $\Sigma\cap \{x_3\geq d-2C\}$ had this property, the previous arguments show that $\Sigma$ would lie in the half-space $\{x_3 \leq d-C\}$, which is not possible since $\Sigma$ intersects by hypothesis the plane $x_3=d$. Thus, $\Sigma\cap \{x_3\leq -d-2C\}$ only has compact connected components, and arguing as in the previous paragraph we deduce that $\Sigma$ is contained in the half-space $\{x_3 \geq -d-3C\}$. Hence, $\Sigma$ is contained in a slab of width $D=2d+4C$ between two planes parallel to $v^{\perp}$. This proves Theorem \ref{t1pe}.
\end{proof}

We provide next some corollaries of this result. 

\begin{cor}\label{2sim}
Let $\cH\in C^1(\S^2)$, $\cH>0$, satisfy properties (a) and (b) of Theorem \ref{t1pe} with respect to two linearly independent directions $v,w\in \S^2$. Then there is some $\alfa=\alfa(\cH,v,w)>0$ such that the following holds: any properly embedded $\cH$-surface of finite topology and one end lies inside a solid cylinder $C(v\wedge w,\alfa)$ of radius $\alfa$ and axis orthogonal to both $v,w$.
\end{cor}

%{\bf (Sobre el corolario de arriba: ¿podemos quitar la hipótesis (b)? Por ejemplo, ¿podemos usar reflexión sobre planos oblicuos?)}

As a particular case of Corollary \ref{2sim}, we have:

\begin{cor}\label{cor:pdd}
Let $\cH\in C^1(\S^2)$, $\cH>0$, be rotationally symmetric, i.e. $\cH(x)=\mathfrak{h}(\esiz x,v\esde)$ for some $v\in \S^2$ and some $\ch\in C^1([-1,1])$.

Then, any properly embedded $\cH$-surface in $\R^3$ of finite topology and one end lies inside a solid cylinder of $\R^3$ with axis parallel to $v$.
\end{cor}
\begin{proof} By Theorem \ref{t1pe}  (or by Corollary \ref{2sim}), it suffices to show that there exist uniform height estimates for $\cH$-graphs in any direction of $\R^3$ orthogonal to $v$. In order to prove this, we can use the argument in the proof of Theorem 6.2 of \cite{EGR}, which we sketch next.

Take $w\in \S^2$ orthogonal to $v$. First observe that, by Lemma \ref{lemi}, in order to prove existence of uniform height estimates for $\cH$-graphs in the $w$-direction, it suffices to do so for \emph{compact} graphs $\Sigma\subset \R^3$ with $\parc \Sigma$ contained in the plane $\Pi=w^{\perp}$, and with the diameter of each connected component of $\parc \Sigma$ being less than $2/H_0$, where $0<H_0<{\rm min}_{\S^2} \cH$. 

Thus, let $\Sigma$ be an such graph, and let $\xi\in \S^2$ be another vector orthogonal to $v$, and which makes an angle of $\pi/4$ with $w$. Note that, as $\cH$ is rotationally symmetric, we can apply the Alexandrov reflection principle with respect to the family of planes of $\R^3$ orthogonal to $\xi$. By the argument in \cite[Theorem 6.2]{EGR} using reflections in this \emph{tilted direction} $\xi$, it can be shown that there exists $C=C(\cH)>0$ (independent of $\Sigma$) such that the distance of each $p\in \Sigma$ to the plane $\Pi$ is bounded by $C$; we omit the specific details. This proves the desired existence of height estimates, and hence, also Corollary \ref{cor:pdd}.

%
%{\bf Hay que probar estimaciones horizontales (tilted planes?)}.
%\textcolor{red}{The corollary will hold once we prove height estimate for compact $\cH$-graphs with boundary lying in a plane containing the vector $v$. Notice that we cannot apply item 3 in Theorem \ref{th:hees} since the existence of a strictly convex $\cH$-sphere is not ensured. However we stand ourselves in position to prove horizontal height estimates following the ideas developed in Theorem 6.2 in [EGR], using tilted planes with respect the plane where the boundary lies. Let $M$ be a compact graph with boundary $\partial M$ contained in a plane containing the vector $v$, $\Pi_v$. Let $a$ be the orthogonal direction to the plane $\Pi_v$. Let $w$ be a vector orthogonal to $v$ and such that $\sphericalangle(v,w)=\pi/4$. Applying similar techniques to the ones in Theorem 6.2 in [EGR] to the family of planes orthogonal to the vector $w$, $\Pi_w$, we ensure the existence of a positive constant $C$ only depending on $\cH$ such that the distance of each $p\in M$ to $\Pi_v$ is bounded by $C$. Alexandrov reflection technique allows us to obtain height estimates for compact, embedded $\cH$-surfaces with boundary lying in the plane $\Pi_v$. We omit the details.} 

\end{proof}

For the case of three reflection symmetries for $\cH$, we obtain a more definite classification result.

\begin{teo}\label{3sim}
Let $\cH\in C^2(\S^2)$, $\cH>0$. Assume that $\cH$ is invariant under three linearly independent geodesic reflections of $\S^2$.%the reflections with respect to three linearly independent geodesics. 

Then, any properly embedded $\cH$-surface $\Sigma_{\cH}$ in $\R^3$ of finite topology and at most one end is the Guan-Guan sphere $S_{\cH}$ associated to $\cH$.
\end{teo}
\begin{proof}
By Theorem \ref{tgg} we know that the Guan-Guan strictly convex $\cH$-sphere $S_{\cH}$ exists. By Corollary \ref{cor:esfe}, there exist uniform height estimates for $\cH$-graphs in any direction $v\in \S^2$. Thus, by Theorem \ref{t1pe} applied to the three directions of symmetry of $\cH$, we conclude that $\Sigma_{\cH}$ lies in a compact region of $\R^3$. Since $\Sigma_{\cH}$ is proper in $\R^3$, $\Sigma_{\cH}$ is compact. By Corollary \ref{3planes}, $\Sigma_{\cH}=S_{\cH}$ (up to translation), what proves the result.
\end{proof}
\begin{cor}\label{cor:3sim}
Let $\cH\in C^2(\S^2)$, $\cH>0$, be rotationally symmetric and even, i.e. $\cH(x)=\mathfrak{h}(\esiz x,v\esde)$ for some $v\in \S^2$ and some $\mathfrak{h}\in C^2([-1,1])$ with $\mathfrak{h}(x)=\mathfrak{h}(-x)$.

Then, any properly embedded $\cH$-surface in $\R^3$ of finite topology and at most one end is the convex rotational $\cH$-sphere $S_{\cH}$ associated to $\cH$, with rotation axis parallel to $v$.
\end{cor}

The regularity in Collorary \ref{cor:3sim} can be weakened to $\cH\in C^1(\S^2)$, by using the existence of rotational $\cH$-spheres proved in \cite{BGM}. Moreover, as mentioned earlier, in \cite{BGM} we show that there exists complete non-entire rotational convex $\cH$-graphs in $\R^3$ for some rotationally symmetric functions $\cH\in C^1(\S^2)$. These examples show that Theorem \ref{3sim} is not true if we only assume that $\cH$ is invariant by two geodesic reflections in $\S^2$ and that Corollary \ref{cor:3sim} is not true if we do not assume that $\cH$ is even.

\section{Stability of $\cH$-surfaces}\label{sec:stable}

\subsection{The stability operator of $\cH$-hypersurfaces}\label{sec:scmc}

This section is devoted to the study of stability of $\cH$-hypersurfaces in $\R^{n+1}$. We start by recalling some basic notions about stability of CMC hypersurfaces. Let $\Sigma^n$ be an immersed, oriented hypersurface in $\R^{n+1}$ with constant mean curvature $H\in \R$. We define its \emph{stability operator} (or \emph{Jacobi operator}) as $\mathcal{L}:=\Delta + |\sigma|^2$, where $\Delta$ denotes the Laplacian of $(\Sigma,g)$ and $|\sigma|$ is the norm of the second fundamental form of $\Sigma$. As $\cL$ is a Schrodinger operator, it is $L^2$ self-adjoint.

The stability operator $\cL$ appears when considering the second variation of the functional ${\rm Area} - n H \, {\rm Volume}$  of which $\Sigma$ is a critical point. We say that $\Sigma$ is \emph{strongly stable} (or simply \emph{stable}) is $-\cL$ is a non-negative operator, that is, $-\int_{\Sigma} f \cL f \geq 0$ for all $f\in C_0^{\8}(\Sigma)$. By a classical criterion by Fischer-Colbrie \cite{FC}, $\Sigma$ is stable if and only if there exists a positive function $u\in C^{\8}(\Sigma)$ with $\cL u \leq 0$.

%A \emph{Jacobi function} on $\Sigma$ is a function $u\in C^{\8}(\Sigma)$ such that $\cL u =0$.

The stability operator $\cL$ also appears as the \emph{linearized mean curvature operator}. Specifically, consider a normal variation 
 \begin{equation}\label{normalvar}
 (p,t)\in \Sigma\times (-\ep,\ep) \mapsto  p+ t f(p) \eta(p),
  \end{equation}
of an immersed oriented hypersurface $\Sigma$ in $\R^{n+1}$, where $\eta:\Sigma\flecha \S^n$ is the unit normal of $\Sigma$ and $f\in C_0^{\8}(\Sigma)$. Then, if for each $t\in(-\ep,\ep)$ we denote by $H(t)$ the mean curvature function of the corresponding surface $\Sigma_t$ given by \eqref{normalvar}, we have 
 \begin{equation}\label{limifo}\cL f = n H'(0).\end{equation}

Let us consider next the case of $\cH$-hypersurfaces in $\R^{n+1}$, for some $\cH\in C^1(\S^n)$, not necessarily constant. In this general situation, there is no known variational characterization of $\cH$-hypersurfaces similar to the one of the CMC case explained above. Still, we can define a stability operator associated to $\cH$-hypersurfaces by considering the linearization of \eqref{presH}, just as in the above characterization of the CMC case. We do this next.

%In the proposition below, we will let $\cH\in C^1(\S^n)$ be considered as a $1$-homogeneous $C^1$ function on $\R^{n+1}\setminus\{0\}$ by $\cH(\landa x)= \landa \cH(x)$ for all $\landa >0$, $x\in \S^n$.

\begin{pro}\label{varcar}
Let $\Sigma$ be an $\cH$-hypersurface in $\R^{n+1}$ for some $\cH\in C^1(\S^n)$, let $f\in C_0^{2}(\Sigma)$, and for each $t$ small enough, let $\Sigma_t$ denote the hypersurface given by \eqref{normalvar}, where $\eta:\Sigma\flecha \S^n$ is the unit normal of $\Sigma$. Denote $\hat{\cH}(t):= H(t)-\cH(\eta^t):\Sigma\times (-\ep,\ep)\flecha \R$, where $H(t)$ and $\eta^t$ stand, respectively, for the mean curvature and the unit normal of $\Sigma_t$. Then,
 \begin{equation}\label{stabh}
n \hat{\mathcal{H}}'(0) = \cL f, \hspace{1cm} \cL f:= \Delta f+ \esiz X_{\cH},\nabla f\esde + |\sigma|^2 f,
 \end{equation}
where $X_{\cH}\in \X(\Sigma)$ is given for each $p\in \Sigma$ by $X_{\cH}(p):=n\nabla_{\S} \cH(\eta(p))$; here $\nabla_{\S}$ denotes the gradient in $\S^n$, and $\Delta,\nabla, |\sigma |$ denote, respectively, the Laplacian, gradient and norm of the second fundamental form of $\Sigma$ with its induced metric.
\end{pro}
\begin{proof}
From now on we work at some fixed but arbitrary $p\in \Sigma$, which will be omitted for clarity reasons. From $\hat{\cH}(t)=H(t)-\cH(\eta^t)$ and \eqref{limifo} we have
 \begin{equation}\label{vc1}
 \def\arraystretch{1.6}\begin{array}{lll}
 n\hat{\mathcal{H}}'(0)&=&  \Delta f + |\sigma|^2 f - n \left. \frac{d}{dt}\right|_{t=0} (\cH(\eta^t))\\ & = & \Delta f + |\sigma|^2 f - n \left \esiz \nabla_{\S} \cH(\eta(p)),\left. \frac{d}{dt}\right |_{t=0} (\eta^t) \right\esde .
 \end{array}
 \end{equation}
Let $\{e_1,\dots, e_n\}$ be a positively oriented orthonormal basis of principal directions in $\Sigma$ at $p$. Then, for $t\in (-\ep,\ep)$, this basis transforms via \eqref{normalvar} to the positively oriented basis $\{e_1^t,\dots, e_n^t\}$ on $\Sigma_t$ given by
 \begin{equation}\label{vc2}
 e_i^t= (1-t f \kappa_i) e_i + t df(e_i) \eta,
 \end{equation}
 where $\kappa_i$ is the principal curvature of $\Sigma$ at $p$ associated to $e_i$. From \eqref{vc2} we get
  $$\left. \frac{d}{dt}\right|_{t=0} (e_1^t\wedge \cdots \wedge e_n^t)=-nHf \eta + \sum_{i=1}^n df(e_i)e_i = -n H f \eta + \nabla f.$$ Noting that $\eta^t=(e_1^t\wedge \cdots \wedge e_n^t)/|e_1^t\wedge \cdots \wedge e_n^t |$, we have
   \begin{equation}\label{vc3}
   \left. \frac{d}{dt}\right|_{t=0} (\eta^t) = \nabla f.
   \end{equation} 
 From \eqref{vc1}, \eqref{vc2}, \eqref{vc3} we obtain \eqref{stabh}. This finishes the proof of Proposition \ref{varcar}.
\end{proof}
Proposition \ref{varcar} justifies the following definition:

\begin{defi}\label{def:stabH}
Let $\Sigma$ be an $\cH$-hypersurface in $\R^{n+1}$ for some $\cH\in C^1(\S^n)$. The \emph{stability operator} of $\Sigma$ is defined as the operator $\cL$ on $\Sigma$ given for each $f\in C_0^{2}(\Sigma)$ by
 \begin{equation}\label{stabo}
\cL f:= \Delta f+ \esiz X_{\cH},\nabla f\esde + |\sigma|^2 f, \hspace{1cm} X_{\cH}(p):=n \nabla_{\S} \cH(\eta(p)).
 \end{equation}
\end{defi}

Note that when $\cH$ is constant, this definition coincides with that of the standard stability operator of CMC hypersurfaces in $\R^{n+1}$ described above.  When $\cH(x)=\esiz x,e_{n+1}\esde$, this notion is also consistent with the usual definition of the stability operator of self-translating solitons of the mean curvature flow, which correspond to the previous choice of $\cH$; see, e.g. \cite{Es,IR,Gu,SX,Gr}.% for some works on self-translating solitons of the mean curvature flow in which this operator is defined.

%{\bf \color{red} (Añadir referencias sobre el operador de Jacobi para solitones de traslación)}.

Since the property of being an $\cH$-hypersurface is invariant by translations of $\R^{n+1}$, we have:

\begin{cor}\label{nueq}
Let $\Sigma$ be an $\cH$-hypersurface in $\R^{n+1}$ for some $\cH\in C^1(\S^n)$, let $\eta:\Sigma\flecha \S^n$ denote its unit normal, choose $a\in \S^n+$, and define $\nu=\esiz \eta,a\esde \in C^{2}(\Sigma)$. Then $\cL \nu=0$, where $\cL$ is the stability operator \eqref{stabo} of $\Sigma$.
\end{cor}
\begin{proof} Consider the variation of $\Sigma$ $$(p,\landa)\in \Sigma\times \R \mapsto p + \landa a$$ and call $\Sigma_{\landa}:=\Sigma + \landa a$. By the implicit function theorem, we can write this variation in a neighborhood of each $(p_0,0)\in \Sigma\times \R$ as a \emph{normal variation} of the form \eqref{normalvar}. Specifically, for any $p$ near $p_0$ in $\Sigma$ and any $t\in (-\ep,\ep)$ with $\ep>0$ small enough we can write
 \begin{equation}\label{comparvar}
 p + t f(p) \eta (p)= \phi(p,t) + \landa(t) a 
 \end{equation}
for some smooth function $f$ defined near $p_0$ on $\Sigma$, and where $\phi(p,t)\in \Sigma$ for all $(p,t)$, and $\landa(t)$ is smooth with $\landa(0)=0$ and $\landa'(0)\neq 0$. By taking the normal component of the derivative of \eqref{comparvar} with respect to $t$ at $t=0$ we obtain using $\esiz \frac{\parc \phi}{\parc t} ,\eta\esde =0$ that
 \begin{equation}\label{landap}
 f=\landa'(0) \esiz \eta,a\esde = \landa'(0) \nu.
 \end{equation}
By Proposition \ref{varcar}, $\cL f=0$, since all the surfaces $\Sigma_{\landa}$ (and consequently all surfaces given by $t={\rm constant}$ in \eqref{comparvar}) have the same prescribed mean curvature $\cH\in C^1(\S^n)$. Thus, by \eqref{landap}, we obtain $\cL \nu =0$ as claimed.
\end{proof}

%(Nota: aclarar este corolario. El campo variacional obtenido por traslaciones {\bf no} es una variación normal de $\Sigma$. Creo que algo está explicado en el paper de Mazzeo y Pacard.)

\subsection{Generalized elliptic operators and stable $\cH$-hypersurfaces}

The following terminology is taken from \cite{Es}:%Motivated by Definition \ref{def:stabH}, we introduce next a general type of linear elliptic operators on Riemannian manifolds that will be of use to us (the terminology here is taken from \cite{Es}):

\begin{defi}
Let $(\Sigma,\esiz,\esde)$ be a Riemannian manifold. A \emph{generalized Schrodinger operator} $L$ on $\Sigma$ is an elliptic operator of the form
 \begin{equation}\label{gensc}
 L=\Delta + \esiz X, \nabla \cdot \esde + q,
 \end{equation}
where $q\in C^{2}(\Sigma)$, $X\in \X(\Sigma)$, and $\Delta,\nabla$ stand for the Laplacian and gradient operators on $\Sigma$.
\end{defi}
If $X=\nabla \phi$ for some $\phi\in C^{2}(\Sigma)$, we say that $L$ is a \emph{gradient Schrodinger operator}. If $X=0$, we get a standard Schrodinger operator. Note that the stability operator \eqref{stabo} for hypersurfaces of prescribed mean curvature $\cH\in C^1(\S^n)$ is a generalized Schrodinger operator.

As explained in Section \ref{sec:scmc}, the stability operator for CMC hypersurfaces is a Schrodinger operator, and so it is $L^2$ self-adjoint. For self-translating solitons of the mean curvature flow (i.e. for the choice $\cH(x):=\esiz x,e_{n+1}\esde$), the stability operator is a \emph{gradient} Schrodinger operator, and so it is $L^2$ self-adjoint with respect to an adequate weighted structure. In both cases, the stability operator comes associated to the second variation of an adequate functional.%, since these classes of surfaces admit a variational characterization. %{\color{blue}This is not true anymore for general choices of $\cH\in C^1(\S^n)$}. 

There also exist some special cases of (non-gradient) generalized Schrodinger operators \eqref{gensc} that appear associated to the second variation of some geometric functionals. This is the case of the stability operator of marginally outer trapped surfaces (usually called, in short, MOTS). See e.g. \cite{AEM,AMS,AM,GaSc}.% for some relevant references on this topic.

We should observe here two key difficulties in working with the stability operator $\cL$ of $\cH$-hypersurfaces in \eqref{stabo}, namely: (i) that $\cL$ is not in general $L^2$ self-adjoint, and (ii) that $\cL$ does not come in general from a variational problem.

Nonetheless, the following definition is natural, as will be discussed below.

\begin{defi}\label{def:stable}
Let $\Sigma$ be an $\cH$-hypersurface in $\R^{n+1}$ for some $\cH\in C^1(\S^n)$, and let $\cL$ denote its stability operator. We say that $\Sigma$ is \emph{stable} if there exists a positive function $u\in C^{2}(\Sigma)$ such that $\cL u \leq 0$.
\end{defi}

By the Fischer-Colbrie theorem \cite{FC}, this notion agrees in the CMC case with the standard definition of stability. The definition also agrees when $\cH(x)=\esiz x,e_{n+1}\esde$ with the usual stability notion for self-translating solutions to the mean curvature flow, see e.g. Proposition 2 in \cite{SX}. Also, the stability notion in Definition \ref{def:stable} is also consistent with the notion of \emph{outermost stability} in MOTS theory, see \cite[Definition 3.1]{AEM}. 

Let us give some further motivation for Definition \ref{def:stable}. Let $L$ be a generalized Schrodinger operator \eqref{gensc} in a Riemannian manifold $\Sigma$, and let $\Omega\subset \Sigma$ denote a smooth bounded domain. Even though $L$ is not formally self-adjoint, it is well known that there exists an eigenvalue $\landa_0$ of $-L$ in $\Omega$ (with Dirichlet conditions) with the smallest real part, that $\landa_0$ is real, and that its corresponding eigenfunction $\psi$ (unique up to multiplication by a non-zero constant) is nowhere zero; see e.g. Section 5 in \cite{AM}. We call $\landa_0$ the \emph{principal eigenvalue} of the operator $L$ on $\Omega$.

By the argument in Remark 5.2 of \cite{AM}, the existence of a positive function $u>0$ on $\Sigma$ such that $Lu\leq 0$ implies that the principal value $\landa_0$ of $-L$ is non-negative on any smooth bounded domain $\Omega\subset \Sigma$. This justifies again the chosen definition for the stability of $\cH$-hypersurfaces in $\R^{n+1}$.

We provide next a geometric interpretation for the stability of $\cH$-hypersurfaces. Let $\Sigma$ denote a stable $\cH$-hypersurface in $\R^{n+1}$, and let $\Omega$ denote a smooth bounded domain of $\Sigma$ whose closure is contained in a larger smooth bounded domain $\Omega'\subset \Sigma$. Then, the principal eigenvalue $\landa_0(\Omega')$ of the stability operator $-\cL$ on $\Omega'$ is non-negative. By the monotonicity of the principal eigenvalue with respect to the inclusion of domains (see e.g. \cite{Pa}), we deduce that $\landa_0(\Omega)>0$. 

Let $\{\Omega_t\}_{t\in (-\ep,\ep)}$ be a smooth variation in $\R^{n+1}$ of the compact (with boundary) $\cH$-hypersurface $\Omega_0:=\overline{\Omega}$, so that the boundary is  fixed by the variation, i.e. $\parc \Omega_t =\parc \Omega$ for every $t$. Assume moreover that all the hypersurfaces $\Omega_t$ are also compact $\cH$-hypersurfaces with boundary (for the same $\cH\in C^1(\S^n)$). By our arguments in Proposition \ref{varcar}, this implies the existence of a function
$u\in C^{2}(\overline{\Omega})$ satisfying $\cL u=0$ on $\Omega$ with $u=0$ on $\parc \Omega$. However, since $\landa_0(\Omega)>0$ by our previous discussion, this function $u$ cannot exist. In other words: \emph{it is not possible to deform a (strictly) stable $\cH$-hypersurface with boundary $\R^{n+1}$ by fixing its boundary, and in a way that all the deformed hypersurfaces in $\R^{n+1}$ have the same prescribed mean curvature $\cH\in C^1(\S^n)$.} This justifies again the notion of stability for $\cH$-hypersurfaces we have introduced.

By Corollary \ref{nueq}, any $\cH$-graph in $\R^{n+1}$ is stable as an $\cH$-hypersurface. More generally, we have:
\begin{cor}
Let $\Sigma$ be an $\cH$-hypersurface in $\R^{n+1}$ for some $\cH\in C^1(\S^n)$, let $\eta:\Sigma\flecha \S^n$ denote its unit normal, and assume that $\esiz \eta,a\esde >0$ for some $a\in \S^n$. Then $\Sigma$ is stable. In particular, any $\cH$-graph is stable.
\end{cor}

For the case of compact $\cH$-hypersurfaces, we have another trivial consequence:

\begin{cor}\label{conoco}
There are no compact (without boundary) stable $\cH$-hypersurfaces in $\R^{n+1}$.
\end{cor}
\begin{proof}
Let $\Sigma$ be a compact $\cH$-hypersurface in $\R^{n+1}$, and let $\cL$ denote its stability operator. By Corollary \ref{nueq}, the kernel of $\cL$ has dimension at least $n+1$, so in particular $0$ is an eigenvalue for $-\cL$ that is not simple (we do not fix Dirichlet conditions here, as $\parc \Sigma$ is empty). Hence, the principal eigenvalue $\landa_0(\Sigma)$ of $-\cL$ in $\Sigma$ is negative. Thus, $\Sigma$ is not stable.
\end{proof}

\subsection{Radius and curvature estimates for stable $\cH$-surfaces}

The next lemma is essentially due to Galloway and Schoen \cite{GaSc}:

\begin{lem}\label{lemgs}
Let $(\Sigma,\esiz,\esde)$ be a Riemannian manifold, and let $L=\Delta +\esiz X,\nabla \cdot\esde + q$ denote a generalized Schrodinger operator on $\Sigma$; here $X\in \X(\Sigma)$ and $q\in C^{0}(\Sigma)$.

Assume that there exists $u\in C^{2}(\Sigma)$, $u>0$, with $Lu \leq 0$. Then the Schrodinger operator
 \begin{equation}\label{simsc}
 \overline{L}:= \Delta + Q, \hspace{1cm} Q:=q-\frac{1}{2} {\rm div}(X) - \frac{|X|^2}{4},
 \end{equation}
satisfies that $-\overline{L}$ is non-negative, i.e. $-\int_{\Sigma} f \overline{L}f \geq 0$ for every $f\in C_0^{2}(\Sigma)$.
\end{lem}
\begin{proof}
The condition $Lu\leq 0$ for $u>0$ can be rewritten as 
\begin{equation}\label{gas1}
\Delta u+\frac{u}{4}|X+2 \nabla\log u |^2-\frac{u}{4} |X|^2- u |\nabla\log u|^2+q u \leq 0.
\end{equation}
As $u>0$ is positive, writing $u=e^{\phi}$ we have from \eqref{gas1} 
$$
\Delta \phi+\frac{1}{4}|X+2\nabla \phi|^2-\frac{|X|^2}{4}+q\leq 0.
$$
Adding and subtracting $\frac{1}{2} \div X$ we have

\begin{equation}\label{gas2}
\div(\nabla \phi + X/2)+|\nabla \phi + X/2|^2-\frac{|X|^2}{4}-\frac{1}{2} \div X+q\leq 0.
\end{equation}
So, if we define now $Y:=\nabla \phi + X/2$ and $Q:=q-\frac{|X|^2}{4}-\frac{1}{2}\div X$, we have from \eqref{gas2}
\begin{equation}\label{gas3}
|Y|^2+Q\leq -\div Y.
\end{equation}
If we consider now $f\in C_0^2(\Sigma)$, then we obtain from \eqref{gas3} 
$$\def\arraystretch{1.3}\begin{array}{rcl}
f^2|Y|^2+f^2Q&\leq& -f^2 \div Y\\
&=&-\div (f^2 Y)+2f\langle\nabla f,Y\rangle\\
&\leq&-\div(f^2 Y)+2|f||\nabla f||Y|\\
&\leq&-\div(f^2 Y)+|\nabla f|^2+f^2|Y|^2.
\end{array}$$
By integrating the above inequality and using that $f$ has compact support, we obtain
$$
\int_{\Sigma}|\nabla f|^2-Qf^2\geq 0,\ \forall f\in C_0^2(\Sigma),
$$
which is equivalent to $-\overline{L}$ being a non-negative operator. This completes the proof.

\end{proof}

%If $M$ is an $\cH$-surface, then $X=(D\cH)^\top$ and $q=|\II|^2$, so the potential $Q$ is
%$$
%Q=|\II|^2-\div (D\cH)^\top-|(D\cH)^\top|^2.
%$$

%Using this operator, we are able to give the first theorem concerning estimates of the length to the boundary of a stable $\cH$-surface:

Let us also recall the next result for $n=2$, implicitly proved in \cite{LR} as a variation of the arguments introduced by Fischer-Colbrie in \cite{FC}; see Theorem 2.8 in \cite{MPR}. %Recall that if $\Sigma$ is a Riemannian manifold and $p\in \Sigma$, then the distance of $p$ to the boundary $\parc \Sigma$, denoted, $d(p,\parc \Sigma)$, is the radius of the largest open geodesic ball centered at $p$ that is contained in the interior of $\Sigma$. {\bf (¿Nos vale esta definición?)}

% {\bf (Decir qué significa $d(p,\parc \Sigma)$)}.

\begin{lem}\label{lemlo}
Let $\Sigma$ denote a Riemannian surface, and assume that the Schrodinger operator $-(\Delta -K+c)$ is non-negative for some constant $c>0$, where $K$ denotes the Gaussian curvature of $\Sigma$. Then for every $p\in \Sigma$ we have $$d(p,\parc \Sigma)\leq \frac{2\pi}{\sqrt{3c}},$$ i.e. the radius of any (open) geodesic ball centered at $p$ that is relatively compact in $\Sigma$ is at most $2\pi/\sqrt{3c}$.
\end{lem}

As a consequence of Lemma \ref{lemgs} and Lemma \ref{lemlo} we can deduce a distance estimate for stable $\cH$-surfaces in $\R^3$, with $\cH\in C^2(\S^2)$. In Theorem \ref{deste} below, $\nabla_{\S}$, $\Delta_{\S}$ and $\nabla^2_{\S}$ denote, respectively, the gradient, Laplacian and Hessian operators on the unit sphere $\S^2$.

We should emphasize regarding Theorem \ref{deste} below that condition \eqref{estrella} in Theorem \ref{deste} cannot be eliminated altogether, or just substituted by the weaker condition $\cH>0$. Indeed, there exist complete strictly convex $\cH$-graphs in $\R^3$ for some choices of $\cH>0$; see \cite{BGM}.

\begin{teo}\label{deste}
Let $\cH\in C^2(\S^2)$ satisfy on $\S^2$ the inequality
 \begin{equation}\label{estrella}
 3\cH^2 + {\rm det} (\nabla_{\S}^2 \cH) + \cH \Delta_{\S} \cH - |\nabla_{\S} \cH|^2 - \frac{1}{4}(\Delta_{\S} \cH)^2 \geq c>0
 \end{equation}
for some constant $c>0$. Then, for every stable $\cH$-surface $\Sigma$ in $\R^3$, and for every $p\in \Sigma$, we have $$d(p,\parc \Sigma)\leq \frac{2\pi}{\sqrt{3c}}.$$
\end{teo}

\begin{proof} Define
\begin{equation}\label{xah}
Q_{\cH}:= |\sigma|^2 -\frac{1}{2}{\rm div}_{\Sigma} (X_{\cH}) - \frac{|X_{\cH}|^2}{4}, \hspace{1cm} X_{\cH}:=2 \nabla_{\S} \cH(\eta) \in \X(\Sigma).
\end{equation}
Here, as usual, $|\sigma|$ denotes the norm of the second fundamental form of $\Sigma$ and $\eta$ its unit normal. Let $\cL$ denote the stability operator of $\Sigma$, defined in \eqref{stabo}. Since $\Sigma$ is stable, it follows by Lemma \ref{lemgs} that the operator $-\overline{\cL}:=-(\Delta + Q_{\cH})$ is non-negative on $\Sigma$. Assume for the moment that
 \begin{equation}\label{desiQ}
Q_{\cH}\geq -K+c
 \end{equation} 
holds, where $K$ is the Gaussian curvature of $\Sigma$. In that case, the operator $-(\Delta -K+c)$ will also be non-negative, and Theorem \ref{deste} will follow directly from Lemma \ref{lemlo}. 

Thus, it only remains to prove \eqref{desiQ}. To do this, we first compute ${\rm div}_{\Sigma}(X_{\cH})$. First, note that if $V\in \X(\Sigma)$ and $\nabla,\overline\nabla$ denote the Riemannian connections of $\Sigma$ and $\R^3$, we have
\begin{equation}\label{dive1}\def\arraystretch{1.5}\begin{array}{lll} \esiz \nabla_V X_{\cH},V\esde &= & 2 \esiz \nabla_V (\nabla_{\S} \cH(\eta)),V\esde = 2 \esiz \overline\nabla_V (\nabla_{\S} \cH(\eta)),V\esde\\
& = & 2 \esiz d(\nabla_{\S} \cH)_{\eta} (d\eta(V)),V\esde = 2 (\nabla_{\S}^2 \cH)_{\eta} (V, d\eta (V)).\end{array}\end{equation} Here, one should observe for the last equality that, since $\esiz \eta,V\esde=0$, then both $V,d\eta(V)$ are tangent to $\S^2$ at the point $\eta$.

Consider now at any $p\in \Sigma$ an orthonormal basis $\{e_1,e_2\}$ of principal directions of $\Sigma$, and let $\kappa_1,\kappa_2$ denote their associated principal curvatures. It follows then from \eqref{dive1} that, at $p$,
 \begin{equation}\label{dive2}
 {\rm div}_{\Sigma} (X_{\cH})= \sum_{i=1}^2 \esiz \nabla_{e_i} X_{\cH}, e_i\esde =-2 \sum_{i=1}^2 \kappa_i\alfa_i, \hspace{1cm} \alfa_i:= (\nabla_{\S}^2 \cH)_{\eta(p)} (e_i, e_i).
 \end{equation}
Then, by \eqref{xah}, \eqref{dive2} and the identity $|\sigma|^2= 4 H_{\Sigma}^2-2K$, we obtain at $p$:

\begin{equation}\label{calQ}
 \def\arraystretch{2.2} \begin{array}{lll}
Q_{\cH}&=& 4H_{\Sigma}^2-2K  - |\nabla_{\S} \cH(\eta)|^2+\kappa_1\alfa_1 +\kappa_2\alfa_2  \\ & = & 3H_{\Sigma}^2 - K +\displaystyle \frac{(\kappa_1 - \kappa_2)^2}{4} - |\nabla_{\S} \cH(\eta)|^2 \\ & & + \displaystyle\frac{(\kappa_1+\kappa_2)(\alfa_1+\alfa_2)}{2} +  \frac{(\kappa_1-\kappa_2)(\alfa_1-\alfa_2)}{2}  \\ & = & 3H_{\Sigma}^2 - K  - |\nabla_{\S} \cH(\eta)|^2 + \displaystyle\left(\frac{\kappa_1-\kappa_2}{2}  + \frac{\alfa_1-\alfa_2}{2}\right) ^2 \\ & & + H_{\Sigma}(\alfa_1+\alfa_2) - \displaystyle \left(\frac{\alfa_1-\alfa_2}{2}\right)^2\end{array}
\end{equation}
Next, note that, by definition of $\alfa_i$ in \eqref{dive2}, we have $\alfa_1+\alfa_2 = \Delta_{\S} \cH(\eta)$, and 
 \begin{equation}\label{desir}
 \left(\frac{\alfa_1-\alfa_2}{2}\right)^2 \leq \frac{(\Delta_{\S} \cH (\eta))^2}{4} - {\rm det}(\nabla_{\S}^2 \cH)_{\eta}.
 \end{equation}
Thus, it follows from \eqref{calQ}, \eqref{desir} and the identity $H_{\Sigma}=\cH(\eta)$ that
 \begin{equation}\label{calQ2}
 Q_{\cH} + K \geq 3 \cH(\eta)^2 + \cH(\eta) \Delta_{\S} \cH(\eta) -\frac{(\Delta_{\S} \cH (\eta))^2}{4} + {\rm det}(\nabla_{\S}^2 \cH)_{\eta} - |\nabla_{\S} \cH(\eta)|^2.
 \end{equation}
Since by hypothesis, $\cH$ satisfies \eqref{estrella} on $\S^2$, we conclude from \eqref{calQ2} that $Q_{\cH}+K \geq c$, which is \eqref{desiQ}. Hence, the proof of Theorem \ref{deste} is complete.

\end{proof}

Let us state two simple remarks that clarify the nature of condition \eqref{estrella}.

\begin{remark}
Let $\cH\in C^2(\S^2)$. Then, there is some $H_0>0$ such that for any $\alfa>H_0$, the function $\cH^*:=\cH+\alfa \in C^2(\S^2)$ satisfies equation \eqref{estrella}.
\end{remark}

\begin{remark}
Let $\cH\in C^2(\S^2)$ satisfy \eqref{estrella}. Then $\cH(x)\neq 0$ at every $x\in \S^2$. Indeed, if $\cH(x)=0$, equation \eqref{estrella} at $x$ turns into $${\rm det}(\nabla_{\S}^2 \cH)_{x}-\frac{(\Delta_{\S} \cH (x))^2}{4}  \geq c+ |\nabla_{\S} \cH(x)|^2>0 ,
$$ which is not possible since the left hand-side of this inequality is always non-positive.
\end{remark}

The next result follows directly from Theorem \ref{deste}. It generalizes the well known fact that there are no complete, stable, non-minimal CMC surfaces in $\R^3$.

\begin{cor}\label{cor:estrella}
Let $\cH\in C^2(\S^2)$ satisfy condition \eqref{estrella}. Then, there are no complete stable $\cH$-surfaces in $\R^3$.
\end{cor}

Recall that any $\cH$-graph in $\R^3$ is stable. Thus, we get as an immediate consequence of Theorem \ref{deste}:
\begin{cor}
Let $\cH\in C^2(\S^2)$ satisfy condition \eqref{estrella}. Then, for any $v\in \S^2$ there exist uniform height estimates for $\cH$-graphs in $\R^3$ in the $v$-direction (see Definition \ref{uhai} for this notion).
\end{cor}

%\begin{enumerate}
%\item
%Hay otros modos de escribir \eqref{estrella}. Indicar alguno.
%\item
%¿Qué sucede en el caso $H(\nu)$? ¿Cuál es la condición?
% \item
%Decir que sea cual sea $\cH$, existe un $H_0>0$ tal que $\cH+H_0$ cumple la condición estrella.
% \item
%¿Existe algún ejemplo con $\cH$ cambiando de signo que cumpla \eqref{estrella}? Parece que no. Hay que ponerlo
% \item
%La estimación nos proporciona, para las $\cH$ que cumplan la condición estrella, una estimación de altura para grafos en cualquier dirección, y por tanto para superficies compactas embebidas en cualquier dirección con respecto a la que $\cH$ sea simétrica (está explicado en Nelli-Rosenberg).
%\end{enumerate}

We next obtain a curvature estimate for stable $\cH$-surfaces in $\R^3$, for the case that $\cH$ satisfies condition \eqref{estrella}. It generalizes to $\cH$-surfaces the classical curvature estimate obtained by Schoen in \cite{Sc} for minimal surfaces; see also Bérard and Hauswirth \cite{BH} for the constant mean curvature case.

\begin{teo}\label{testicu}
Let $a,c>0$. Then, exists a constant $C=C(a,c)>0$ such that the following statement is true: if $\cH\in C^2(\S^2)$ satisfies 
 \begin{equation}\label{boundH}
|\nabla_{\S} \cH | + |\nabla_{\S}^2 \cH| \leq a
 \end{equation} on $\S^2$, and $\Sigma$ is any stable $\cH$-surface in $\R^3$ satisfying \eqref{estrella} for the constant $c$, then for any $p\in \Sigma$ the following estimate holds:
$$
|\sigma(p)|d_{\Sigma}(p,\partial \Sigma)\leq C.
$$
\end{teo}

\begin{proof}
Arguing by contradiction, assume that there exists a sequence of functions $\cH_n\in C^2(\S^2)$ satisfying \eqref{boundH} for the constant $a$, a sequence $\Sigma_n$ of stable $\cH_n$-surfaces that satisfy \eqref{estrella} for the constant $c$, and points $p_n\in\Sigma_n$ satisfying 
\begin{equation}\label{cest1}
|\sigma_n(p_n)|d_{\Sigma_n}(p_n,\partial\Sigma_n)>n,
\end{equation}
where $\sigma_n$ denotes the second fundamental form of $\Sigma_n$. First, note that \eqref{cest1} implies that $|\sigma_n(p_n)|$ diverges to $\8$, since by Theorem \ref{deste} we have for all $n$ a uniform upper bound $d_{\Sigma_n}(p_n,\partial\Sigma_n)\leq d$ for some $d>0$. %Passing to a subsequence, we will assume that
%\begin{equation}\label{cest2}
%|\sigma_n(p_n)|>n,\ \forall n\in\mathbb{N}.
%\end{equation}
At this point, we can use the arguments in the proof of Theorem \ref{th:curv}, with some modifications.
% and show that a subsequence of adequate dilations in $\R^3$ of the surfaces $\Sigma_n$ converges uniformly on compact sets in the $C^2$ topology to a complete minimal surface $\Sigma_{\8}$ in $\R^3$ of bounded curvature, that passes through the origin, and has norm of the second fundamental form equal to $1$ at the origin.
%We claim that $\Sigma_{\8}$ is stable. Otherwise, there would exist a compact smooth, stricty unstable domain $\Omega_{\8}\subset \Sigma_{\8}$.
%*******
Let $d_n:=d_{\Sigma_n}(p_n,\partial\Sigma_n)$, and consider the compact intrinsic balls $D_n=B_{\Sigma_n}(p_n,d_n/2)$. Let $q_n\in D_n$ be the maximum of the function 
$$
h_n(x)=|\sigma_n(x)|d_{\Sigma_n}(x,\partial D_n)
$$
on $\overline{D_n}$. Clearly $q_n$ is an interior point of $D_n$, as $h_n$ vanishes on $\parc D_n$. Then, if we denote $\lambda_n=|\sigma_n(q_n)|$ and
$r_n=d_{\Sigma_n}(q_n,\partial D_n)$, we have by \eqref{cest1}:
\begin{equation}\label{cest3}
\landa_n r_n = h_n(q_n) \geq h_n(p_n)  \to \8 \hspace{0.5cm} \text{(as $n\to \8$)}.
\end{equation}
Let now $H_n$ denote the maximum of the mean curvature function of $\Sigma_n$ restricted to $B_{\Sigma_n}(q_n,r_n/2)$, and define $\lambda^*_n:=\max\{\lambda_n, H_n\}$. Consider the rescalings by $\landa_n^*$ of the immersed surfaces $B_{\Sigma_n}(q_n,r_n/2)\subset\Sigma_n$, in a similar fashion to the argument in Theorem \ref{th:curv}, and denote these rescalings by $M_n:=\lambda^*_n B_{\Sigma_n}(q_n,r_n/2)$. Let $q_n^*$ be the point in $M_n$ that corresponds to $q_n\in \Sigma_n$, and consider the translated surfaces $M_n^*:=M_n-q_n^*$ which take the points $q_n^*$ to the origin. Note that the distance in $M_n^*$ from the origin to $\parc M_n^*$ is equal to $\landa^*_n r_n/2$, and so it diverges to $\8$ as $n\to \8$,  by \eqref{cest3}.

The surfaces $M_n^*$ have uniformly bounded second fundamental form, since for any $z_n\in B_{\Sigma_n}(q_n,r_n/2)$ we have 
$$
\frac{|\sigma_n(z_n)|}{\lambda^*_n}\leq \frac{|\sigma_n(z_n)|}{\lambda_n} =\frac{h_n(z_n)}{\lambda_n d_{\Sigma_n}(z_n,\partial D_n)}\leq\frac{d_{\Sigma_n}(q_n,\partial D_n)}{d_{\Sigma_n}(z_n,\partial D_n)}\leq 2,
$$ where in the last inequality we have used \eqref{destr}, that also holds in the present context.

Also, each $M_n^*$ is a surface of prescribed mean curvature $\cH_n^*\in C^2(\S^2)$, where $\cH_n^*(x):=\cH_n(x)/\landa_n^*$. By definition of $\landa_n^*$, we have $\cH_n^* \leq 1$ on the Gauss map image $\Omega_n\subset \S^2$ of $M_n^*$ in $\S^2$, for all $n$. Also, note that as $n\to \8$ we have by \eqref{boundH} that \begin{equation}\label{boundH2} |\nabla_{\S} \cH_n^* | + |\nabla_{\S}^2 \cH_n^* | \to 0
 \end{equation}
globally on $\S^2$. Consequently, a subsequence of the $\cH_n^*$ converges in the $C^1(\S^2)$ topology to a constant $\cH_{\8} \in [0,1]$. It follows then from Theorem \ref{compa} that a subsequence of the $M_n^*$ converges uniformly on compact sets in the $C^3$ topology to a complete surface $\Sigma^*$ in $\R^3$ of constant mean curvature $\cH_{\8}$ that passes through the origin. We consider the connected component of $\Sigma^*$ that passes through the origin, which will still be denoted by $\Sigma^*$.

Since each $\Sigma_n$ is stable, it follows that each $M_n^*$ is a stable $\cH_n^*$-surface in $\R^3$. So, by Lemma \ref{lemgs} (see also the beginning of the proof of Theorem \ref{deste}), we see that the Schrodinger operator $-\overline{\cL}_n :=-(\Delta_n + Q_{\cH^*_n})$ is non-negative on $M_n^*$; here $\Delta_n$ denotes the Laplacian operator in $M_n^*$, and $Q_{\cH^*_n}$ is given by \eqref{xah}. 

Clearly, the Schrodinger operators $\overline{\cL}_n$ converge to the Jacobi operator $\cL_{\8}$ of the CMC surface $\Sigma^*$. It follows then that $-\cL_{\8}$ is non-negative, i.e. $\Sigma^*$ is stable. As $\Sigma^*$ is also complete, then necessarily $\cH_{\8}=0$ and $\Sigma^*$ is a plane. In particular, the norm of the second fundamental forms of the $M_n^*$ at the origin converge to zero, which implies that $\landa_n/\landa_n^*\to 0$. In particular, by the definition of $\landa_n^*$, for $n$ large enough we have $\landa_n^*=H_n$. This implies that $\cH_n^*(x_n)=1$ for some $x_n\in \S^2$, for each $n$ large enough. This is a contradiction with $\cH_{\8}=0$, which completes the proof of Theorem \ref{testicu}.

\end{proof}

\def\refname{References}

\hspace{0.4cm}

\noindent The authors were partially supported by
MICINN-FEDER, Grant No. MTM2016-80313-P, Junta de Andalucía Grant No.
FQM325, 
and Programa de Apoyo
a la Investigacion, Fundacion Seneca-Agencia de Ciencia y
Tecnologia Region de Murcia, reference 19461/PI/14.

\end{document}